\documentclass[10pt,reqno]{amsart}
\pdfoutput=1
\usepackage{amsfonts}
\usepackage{mathrsfs}  
\usepackage{amsthm}
\usepackage{amsxtra}
\usepackage{amssymb}
\usepackage{tikz}
\usepackage{caption}
\usepackage{graphicx}
\usepackage{subfigure}
\usepackage{multirow,booktabs}
\usepackage{float} 
\usepackage[pagewise]{lineno}
\newtheorem{theorem}{Theorem}[section]

\newtheorem{lemma}[theorem]{Lemma}
\newtheorem{proposition}[theorem]{Proposition}
\newtheorem{assumption}[theorem]{Assumption}
\newtheorem{remark}[theorem]{Remark}

\newtheorem{algorithm}{Algorithm}
\newcommand*{\dif}{\mathop{}\!\mathrm{d}}

\def\EE{\mathbb E}
\def\P{\mathbb P}
\def\R{{\mathcal R}}
\def\d{\delta}
\def\l{\langle}
\def\p{\partial}
\def\r{\rangle}
\def\F{{\mathcal F}}

\def\X{{\mathcal X}}
\def\Y{{\mathcal Y}}

\begin{document}

\title[Randomized block coordinate descent method]{Randomized block coordinate descent method for linear ill-posed problems}
\author[Qinian Jin]{Qinian Jin} 
\address[]{Mathematical Scuences Institute, Australian National University, Canberra, ACT 2601, Australia}
\email[]{qinian.jin@anu.edu.au}
\author[Duo Liu]{Duo Liu}
\address[]{School of Mathematics and Statistics, Beijing Jiaotong University, Beijing 100044, China}
\email[]{20118001@bjtu.edu.cn}

%\author{author\\university} 
%\CTEXoptions[today=old]	
%\maketitle	
\begin{abstract}
Consider the linear ill-posed problems of the form $\sum_{i=1}^{b} A_i x_i =y$, where, for each $i$, 
$A_i$ is a bounded linear operator between two Hilbert spaces $X_i$ and $\Y$. When $b$ is huge, 
solving the problem by an iterative method using the full gradient at each iteration step is both 
time-consuming and memory insufficient. Although randomized block coordinate decent (RBCD) method 
has been shown to be an efficient method for well-posed large-scale optimization problems with 
a small amount of memory, there still lacks a convergence analysis on the RBCD method for solving  
ill-posed problems. In this paper, we investigate the convergence property of the RBCD method with noisy 
data under either {\it a priori} or {\it a posteriori} stopping rules. We prove that the RBCD method 
combined with an {\it a priori} stopping rule yields a sequence that converges weakly to a solution 
of the problem almost surely. We also consider the early stopping of the RBCD method and 
demonstrate that the discrepancy principle can terminate the iteration after finite many steps almost 
surely. For a class of ill-posed problems with special tensor product form, we obtain strong convergence 
results on the RBCD method. Furthermore, we consider incorporating the convex regularization terms 
into the RBCD method to enhance the detection of solution features. To illustrate the theory and the performance 
of the method, numerical simulations from the imaging modalities in computed tomography and compressive 
temporal imaging are reported.
\end{abstract}



\maketitle

\section{\bf Introduction}

In this paper we consider linear ill-posed problems of the form
\begin{equation}\label{eq:problem}
\sum_{i=1}^{b} A_i x_i =y,
\end{equation}
where, for each $i=1,\cdots,b$, $A_i:X_i\to \Y$ is a bounded linear operator between two real 
Hilbert spaces $X_i$ and $\Y$. Ill-posed problems with such a form arise naturally from many applications in imaging. For example, coded
aperture compressive temporal imaging (more details see Example \ref{subsect4.2} and \cite{LLY2013}), a kind of large-scale snapshot compressive imaging \cite{YBK2021}, can be expressed by the following integral 
equation 
\begin{equation}\label{eq:CACTI}
(Ax)(s):=\int_{T}   m(s,t)x(s,t) \dif t = y(s), \quad s \in \Omega,  
\end{equation}
where $T=[t_1,t_2] \subset [0,+\infty)$, $\Omega$ is a bounded domains in 2-dimensional Euclidean spaces and $m$ is a bounded 
continuous mask function on $\Omega\times T$, which ensures the forward operator $A$ is a bounded linear operator from $L^2(\Omega\times T)$ to $L^2(\Omega)$. Decomposing the whole time
domain $T$ into $b$ disjoint sub-domains $T_1, \ldots, T_b$, where for any $i=1,\ldots,b$,
\begin{equation}\label{eq:CACTI_1}
T_i := [\bar t_i, \bar t_{i+1}], \quad \bar t_i := \frac{ t_2 - t_1}{b}(i-1) + t_1,    
\end{equation}
we denote the restriction of each function $x \in L^2(\Omega\times T)$ to $\Omega\times T_i$ as $x_i$, i.e. $x_i:=x|_{\Omega\times T_i}$ and let
\begin{equation}\label{eq:CACTI_2}
(A_i x_i)(s):=\int_{T_i}m(s,t)x_i(s,t) \dif t, \quad s \in \Omega,    
\end{equation}
then $A_i:L^2(\Omega\times T_i)\rightarrow L^2(\Omega)$ is a bounded linear operator for each $i$. 
Since
$$
\int_T m(s,t)x(s,t)\dif t = \sum_{i=1}^{b} \int_{T_i}m(s,t) x_i(s,t)\dif t,
$$
the above integral equation can be written as the form (\ref{eq:problem}). For more examples, one may 
refer to the light field reconstruction from the focal stack \cite{LQJ2017,YWL2016} and other imaging 
modalities such as spectral imaging \cite{WJW2008} and full 3-D computed tomography \cite{N2001}. In these 
examples, the sought solution can be separated into many films or frames naturally.

Let $\X := X_1\times\cdots\times X_b$ be the product space of $X_1, \cdots, X_b$ with the natural 
inner product inherited from those of $X_i$ and define $A: \X\to \Y$ by
\begin{equation}\label{opA}
Ax := \sum_{i=1}^b A_i x_i, \quad \forall x=(x_1,\ldots,x_b)\in \X. 
\end{equation}
Then, (\ref{eq:problem}) can be equivalently stated as $Ax = y$. Throughout the paper we assume 
that (\ref{eq:problem}) has a solution, i.e. $y \in \mbox{Ran}(A)$, the range of $A$. 
In practical applications, instead of the exact data $y$ we can only acquire a noisy data 
$y^\d$ satisfying
$$
\|y-y^\d\|\le \d,
$$
where $\d>0$ is the noise level. How to stably reconstruct an approximate solution of ill-posed 
problems described by (\ref{eq:problem}) using noisy data $y^\d$ is a central topic in computational 
inverse problems. In this situation, the regularization techniques should be used and various 
regularization methods have been developed in the literature (\cite{EHN1996}). 

In this paper we will exploit the decomposition structure of the equation (\ref{eq:problem}) 
to develop a randomized block coordinate descent method for solving ill-posed linear problems. 
The researches on the block coordinate descent method in the literature mainly focus on solving 
large scale well-posed optimization problems (\cite{LX2015,N2012,RT2014,ST2013,T2001,TY2009a,TY2009b,W2015}).
For the unconstrained minimization problem 
\begin{align}\label{min}
\min_{(x_1, \cdots, x_b) \in \X} f(x_1, \cdots, x_b)
\end{align}
with $\X = X_1 \times \cdots \times X_b$, where $f: \X \to {\mathbb R}$ 
is a continuous differentiable function, the block coordinate descent method updates 
$x_{k+1} := (x_{k+1,1}, \cdots, x_{k+1, b})$ from $x_k := (x_{k,1}, \cdots, x_{k,b})$ 
by selecting an index $i_k$ from $\{1, \cdots, b\}$ and uses the partial gradient 
$\nabla_{i_k} f(x_k)$ of $f$ at $x_k$ to update the component $x_{k+1, i_k}$ of $x_{k+1}$ 
and leave other components unchanged; more precisely, we have the updating formula
\begin{equation}\label{BCD}
x_{k+1,i} = 
\begin{cases}
x_{k,i_k} - \gamma \nabla_{i_k} f(x_k), & \text{if} \ i =  i_k, \\
x_{k,i}, & \text{otherwise}, 
\end{cases}
\end{equation}
where $\gamma$ is the step size. The index $i_k$ can be selected in various ways, either in 
a cyclic fashion or in a random manner. By applying the block coordinate descent method (\ref{BCD})
to solve (\ref{min}) with $f$ given by 
\begin{align}\label{res}
f(x_1, \cdots, x_b) = \frac{1}{2} \left\|\sum_{i=1}^b A_i x_i - y^\d\right\|^2,
\end{align}
it leads to the iterative method 
\begin{equation}\label{eq:RBCD}
x_{k+1,i}^{\delta} = 
\begin{cases}
x_{k,i_k}^{\delta}-\gamma A^{*}_{i_k}(Ax_{k}^{\delta}-y^{\delta}), &\text{if} \ i =  i_k, \\
x_{k,i}^{\delta}, &\text{otherwise}
\end{cases}
\end{equation}
for solving ill-posed linear problem (\ref{eq:problem}) using noisy data $y^\d$, where $A_i^*: \Y \to X_i$ 
denotes the adjoint of $A_i$ for each $i$ and the superscript ``$\d$" on all iterates is used to indicate their dependence on the noisy data. 

The method (\ref{eq:RBCD}) can also be motivated from the Landweber iteration
\begin{equation}\label{eq:Landweber_itration}
 x_{k+1}^\d = x_{k}^{\d} - \gamma A^*(Ax_k^\d - y^\d),   
\end{equation}
which is one of the most prominent regularization methods for solving ill-posed problems,  
where $\gamma$ is the step-size and $A^*: \Y \to \X$ is the adjoint of $A$. Note that, with the operator 
$A$ defined by (\ref{opA}), we have
\begin{equation*}
A^* z = \left(A_1^* z, \cdots, A_b^* z\right), \quad \forall z \in \Y. 
%(Ax_k^\d - y^\d) = \left(A_1^*(Ax_k^\d - y^\d), \cdots, A_b^*(Ax_k^\d - y^\d)\right).
\end{equation*}
Thus, updating $x_{k+1}^\d$ from $x_k^\d$ by (\ref{eq:Landweber_itration}) needs calculating 
$A_i^*(Ax_k^\d-y^\d)$ for all $i=1,\cdots,b$. If $b$ is huge, applying the method 
(\ref{eq:Landweber_itration}) to solve the equation (\ref{eq:problem}) is both time-consuming 
and memory insufficient. To resolve these issues, it is natural to calculate only one component 
of $A^*(A x_k^\d- y^\d)$ and use it to update the corresponding component of $x_{k+1}^\d$ and 
keep other components fixed. This leads again to the block coordinate descent method (\ref{eq:RBCD}) 
for solving (\ref{eq:problem}). 

The block coordinate descent method and its variants for large scale well-posed optimization problems have been 
analyzed extensively, and all the established convergence results require either the objective function 
to be strongly convex or the convergence is established in terms of the objective value. However, these 
results are not applicable to the method (\ref{eq:RBCD}) for ill-posed problems because the corresponding 
objective function (\ref{res}) is not strongly convex, and moreover, due to the ill-posedness of the 
underlying problem, convergence on objective value does not imply any result on the iterates directly. Therefore, new analysis is required for understanding the method (\ref{eq:RBCD}) for ill-posed problems. 

In \cite{RNH2019} the block coordinate descent method (\ref{eq:RBCD}) was considered with the 
index $i_k$ selected in a cyclic fashion, i.e. $i_k= (k\ \text{mod}\ b) +1$. The regularization property 
of the corresponding method was only established for a specific class of linear ill-posed problems in 
which the forward operator $A$ is assumed a particular tensor product form; more precisely, the ill-posed 
linear system of the form 
\begin{align}\label{prob0}
\sum_{i=1}^b v_{li} K x_i = y_l, \quad l = 1, \cdots, d, 
\end{align}
was considered in \cite{RNH2019}, where $K: X \to Y$ is a bounded linear operator between two real Hilbert 
spaces $X$ and $Y$ and $V =(v_{li})$ is a $d\times b$ matrix with full column rank. Clearly this problem 
can be cast into the form (\ref{eq:problem}) if we take $X_i = X$ for each $i$, $\Y = Y^d$, the $d$-fold 
Cartesian product of $Y$, and define $A_i: X \to \Y$ by 
$$
A_i z = (v_{1i} K z, \cdots, v_{di} K z), \quad z \in X
$$
for each $i = 1, \cdots, b$. The problem formulation (\ref{prob0}), however, seems too specific to cover interesting 
problems arising from practical applications. How to establish the regularization property of this 
cyclic block coordinate descent method for ill-posed problem (\ref{eq:problem}) in its full generality 
remains a challenging open question. 

%Moreover, it is noted that the performance of this method relies on the order of sequence of the coordinate blocks chosen in practice. To handle these issues both in theory and applications, 
Instead of using a deterministic cyclic order, in this paper we will consider a randomized block coordinate 
descent method, i.e. the method (\ref{eq:RBCD}) with $i_k$ being selected from $\{1, \cdots, b\}$ uniformly 
at random. We will investigate the convergence behavior of the iterates for the ill-posed linear problem 
(\ref{eq:problem}) with general forward operator $A$. Note that in the method (\ref{eq:RBCD}) the definition 
of $x_{k+1}^\d$ involves $r_k^\d := A x_k^\d - y^\d$. If this quantity is required to be calculated directly 
at each iteration step, then the advantages of the method will be largely reduced. Fortunately, this quantity 
can be calculated in a simple recursive way. Indeed, by the definition of $x_{k+1}^\d$ we have 
\begin{align*}
r_{k+1}^\d & = \sum_{i=1}^b A_i x_{k+1}^\d - y^\d 
= \sum_{i\ne i_k} A_i x_{k+1, i}^\d + A_{i_k} x_{k+1, i_k}^\d - y^\d \\
& = \sum_{i\ne i_k} A_i x_{k,i}^\d + A_{i_k} x_{k,i_k}^\d + A_{i_k} (x_{k+1,i_k}^\d - x_{k,i_k}^\d) - y^\d \\
& = A x_k^\d - y^\d + A_{i_k}(x_{k+1, i_k}^\d - x_{k, i_k}^\d) \\
& = r_k^\d + A_{i_k} (x_{k+1, i_k}^\d - x_{k, i_k}^\d).
\end{align*}
Therefore, we can formulate our randomized block coordinate descent (RBCD) method into the following 
equivalent form which is convenient for numerical implementation.

\begin{algorithm}[RBCD method for ill-posed problems]\label{alg:RBCD}
Pick an initial guess $x_0 \in \X$, set $x_0^\d := x_0$ and calculate $r_0^\d:= A x_0^\d - y^\d$. 
Choose a suitable step size $\gamma>0$. For all integers $k \ge 0$ 
do the following:  
\begin{enumerate}
\item[(i)] Pick an index $i_k \in \{1, \cdots, b\}$ randomly via the uniform distribution; 

\item[(ii)] Update $x_{k+1}^\d$ by setting $x_{k+1,i}^\d = x_{k,i}^\d$ for $i \ne i_k$ and 
$$
x_{k+1, i_k}^\d = x_{k, i_k}^\d - \gamma A_{i_k}^* r_k^\d;
$$

\item[(iii)] Calculate $r_{k+1}^\d = r_k^\d + A_{i_k}(x_{k+1, i_k}^\d - x_{k,i_k}^\d)$. 
\end{enumerate}
\end{algorithm}

In this paper we present a convergence analysis on Algorithm \ref{alg:RBCD} by deriving the 
stability estimate and establishing the regularization property. For the equation (\ref{eq:problem}) 
in the general form, we obtain a weak convergence result and for a particular tensor product form 
as studied in \cite{RNH2019} we establish the strong convergence result. Moreover, we also consider 
the early stopping of Algorithm \ref{alg:RBCD} and demonstrate that the discrepancy principle 
can terminate the iteration after finite many steps almost surely. Note that, Algorithm \ref{alg:RBCD} 
does not incorporate {\it a priori} available information on the feature of the sought solution. 
In case the sought solution has some special 
feature, such as non-negativity, sparsity and piece-wise constancy, Algorithm \ref{alg:RBCD} may not 
be efficient enough to produce satisfactory approximate solutions. In order to resolve this issue, 
we further extend Algorithm \ref{alg:RBCD} by incorporating into it a convex regularization term 
which is selected to capture the desired feature. For a detailed description of this extended 
algorithm, please refer to Algorithm \ref{alg:RBCD+} for which we will provide a convergence analysis. 

It should be pointed out that the RBCD method is essentially different from the Landweber-Kaczmarz method 
and its stochastic version which have been studied in \cite{HLS2007,Jin2016,JLZ2023,JW2013,KS2002}.
The Landweber-Kaczmarz method relies on decomposing the data into many blocks of small size, while our 
RBCD method makes use of the block structure of the sought solutions. 

This paper is organized as follows. In Section \ref{sect2} we consider Algorithm \ref{alg:RBCD} by 
proving various convergence results and investigating the discrepancy principle as an early stopping 
rule. In Section \ref{sect3} we consider an extension of Algorithm \ref{alg:RBCD} by incorporating 
a convex regularization term into it so that the special feature of sought solutions can be detected. 
Finally in Section \ref{sect4} we provide numerical simulations to test the performance of our
proposed methods.

\section{\bf Convergence analysis}\label{sect2}

In this section we consider Algorithm \ref{alg:RBCD}. We first 
establish a stability estimate, and prove a weak convergence result when the method is 
terminated by an {\it a priori} stopping rule. We then investigate the discrepancy principle 
and demonstrate that it can terminate the iteration in finite many steps almost surely. 
When the forward operator $A$ has a particular tensor product form as used in \cite{RNH2019},
we further show that strong convergence can be guaranteed. 

Note that, once $x_0\in \X$ and $\gamma$ are fixed, the sequence $\{x_k^\d\}$ is completely 
determined by the sample path $\{i_0, i_1, \cdots\}$; changing the sample path can result in a 
different iterative sequence and thus $\{x_k^\d\}$ is a random sequence. Let $\F_0 = \emptyset$ 
and, for each integer $k \ge 1$, let $\F_k$ denote the $\sigma$-algebra generated by the random 
variables $i_l$ for $0\le l < k$. Then $\{\F_k: k\ge 0\}$ form a filtration which is natural 
to Algorithm \ref{alg:RBCD}. Let $\EE$ denote the expectation associated with this filtration, 
see \cite{B2020}. The tower property of conditional expectation 
$$
\EE[\EE[\phi|\F_k]] = \EE[\phi] \quad \mbox{ for any random variable } \phi
$$
will be frequently used. 

\subsection{\bf Stability estimate} 

Let $\{x_k\}$ denote the iterative sequence produced by Algorithm \ref{alg:RBCD} 
with $y^\d$ replaced by the exact data $y$. By definition it is easy to see that, for 
any fixed integer $k \ge 0$, along any sample path there holds $\|x_k^\d - x_k\| \to 0$ 
as $\d \to 0$ and hence $\EE[\|x_k^\d - x_k\|^2] \to 0$ as $\d \to 0$. The following 
result gives a more precise stability estimate.  

\begin{lemma}\label{RBCD.lem1}
Consider Algorithm \ref{alg:RBCD} and assume $0<\gamma \le 2/\|A\|^2$. Then, 
along any sample path there holds 
\begin{align}\label{RBCD.-2}
\|A(x_k^\d - x_k) - y^\d + y\| \le \d 
\end{align}
for all integers $k \ge 0$. Furthermore 
\begin{align}\label{RBCD.-1}
\EE[\|x_k^\d - x_k\|^2] \le \frac{2\gamma}{b} k \d^2 
\end{align}
for all $k \ge 0$.
\end{lemma}

\begin{proof}
Let $u_k^\d:= x_k^\d - x_k$ for all integers $k \ge 0$. It follows from the definition of $x_{k+1}^\d$ 
and $x_{k+1}$ that 
\begin{align*}
u_{k+1, i}^\d = \left\{\begin{array}{lll}
u_{k, i}^\d    & \mbox{ if } i \ne i_k,\\[1.1ex]
u_{k,i_k}^\d - \gamma A_{i_k}^* (A u_k^\d - y^\d + y) & \mbox{ if } i = i_k. 
\end{array}\right.
\end{align*}
Thus
\begin{align*}
A u_{k+1}^\d - y^\d + y 
& = \sum_{i\ne i_k} A_i u_{k+1,i}^\d + A_{i_k} u_{k+1, i_k}^\d -y^\d +y \\
& = \sum_{i=1}^b A_i u_{k,i}^\d - \gamma A_{i_k} A_{i_k}^*(A u_k^\d - y^\d +y) - y^\d + y\\
& = (I - \gamma A_{i_k} A_{i_k}^*) (A u_k^\d - y^\d + y). 
\end{align*}
Since $\|A_{i_k}\| \le \|A\|$ and $0< \gamma \le 2/\|A\|^2$, we have 
$\|I - \gamma A_{i_k} A_{i_k}^*\| \le 1$ and thus
$$
\|A u_{k+1}^\d - y^\d + y\| \le  \|I - \gamma A_{i_k} A_{i_k}^*\| \|A u_k^\d - y^\d + y\|
\le \|A u_k^\d - y^\d + y\|.
$$
Consequently 
\begin{align*}
\|A u_{k+1}^\d - y^\d + y\| \le \|A u_0^\d - y^\d +y\| = \|y^\d - y\| \le \d
\end{align*}
which shows (\ref{RBCD.-2}). 

To derive (\ref{RBCD.-1}), we note that 
\begin{align*}
\|u_{k+1}^\d\|^2 
& = \sum_{i\ne i_k} \|u_{k+1, i}^\d\|^2 + \|u_{k+1, i_k}^\d\|^2 \\
& = \sum_{i\ne i_k} \|u_{k,i}^\d\|^2 + \|u_{k,i_k}^\d - \gamma A_{i_k}^* (A u_k^\d - y^\d + y)\|^2 \displaybreak[0]\\
& = \sum_{i\ne i_k} \|u_{k,i}^\d\|^2 + \|u_{k,i_k}^\d\|^2 + \gamma^2 \|A_{i_k}^* (A u_k^\d - y^\d + y)\|^2 \displaybreak[0]\\
& \quad \, - 2 \gamma \l u_{k,i_k}^\d, A_{i_k}^*(A u_k^\d - y^\d + y)\r \displaybreak[0]\\
& = \|u_k^\d\|^2 + \gamma^2 \|A_{i_k}^*(A u_k^\d - y^\d + y)\|^2 \displaybreak[0]\\
& \quad \, - 2\gamma \l A_{i_k} u_{k,i_k}^\d, A u_k^\d - y^\d + y\r. 
\end{align*}
Taking the conditional expectation gives 
\begin{align*}
& \EE[\|u_{k+1}^\d\|^2| \F_k] \\
& = \|u_k^\d\|^2 + \frac{\gamma^2}{b} \sum_{i=1}^b \|A_i^* (A u_k^\d - y^\d + y)\|^2 
- \frac{2\gamma}{b} \left\l\sum_{i=1}^b A_i u_{k,i}^\d, A u_k^\d - y^\d + y\right\r \displaybreak[0]\\
& = \|u_k^\d\|^2 + \frac{\gamma^2}{b} \|A^* (A u_k^\d - y^\d + y)\|^2 
- \frac{2\gamma}{b} \left\l A u_k^\d, A u_k^\d - y^\d + y\right\r \displaybreak[0]\\
& \le  \|u_k^\d\|^2 + \frac{\gamma^2 \|A\|^2}{b} \|A u_k^\d - y^\d + y\|^2 
- \frac{2\gamma}{b} \|A u_k^\d - y^\d + y\|^2 \displaybreak[0]\\
& \quad \, - \frac{2\gamma}{b} \l y^\d - y, A u_k^\d - y^\d + y\r \displaybreak[0]\\
& \le \|u_k^\d\|^2 - \frac{1}{b}(2-\gamma\|A\|^2) \gamma \|A u_k^\d - y^\d + y\|^2 
+ \frac{2\gamma}{b} \d \|A u_k^\d - y^\d + y\|.
\end{align*}
By using $0<\gamma \le 2/\|A\|^2$ and (\ref{RBCD.-2}), we further obtain 
\begin{align*}
\EE[\|u_{k+1}^\d\|^2|\F_k] 
\le \|u_k^\d\|^2  + \frac{2\gamma}{b} \d \|A u_k^\d - y^\d + y\|
\le \|u_k^\d\|^2  + \frac{2\gamma}{b} \d^2.
\end{align*}
Consequently, by taking the full expectation and using the tower property of conditional expectation, 
we can obtain 
\begin{align*}
\EE\left[\|u_{k+1}^\d\|^2\right] = \EE\left[\EE\left[\|u_{k+1}^\d\|^2|\F_k\right]\right]  
\le \EE\left[\|u_k^\d\|^2\right] + \frac{2\gamma \d^2}{b}.
\end{align*}
Based on this inequality and the fact $u_0^\d =0$, we can use an induction argument to complete 
the proof of (\ref{RBCD.-1}) immediately. 
\end{proof}

\subsection{\bf Weak convergence}\label{sect2.2}

Our goal is to investigate the approximation behavior of $x_k^\d$ to a solution of $A x = y$. 
Because of Lemma \ref{RBCD.lem1}, we now focus our consideration on the sequence $\{x_k\}$ 
defined by Algorithm \ref{alg:RBCD} using exact data. 

\begin{lemma}\label{RBCD.lem2}
Assume $0< \gamma < 2/\|A\|^2$. Then for any solution $\bar x$ of (\ref{eq:problem}) there holds 
\begin{align}\label{RBCD.1}
\EE[\|x_{k+1} - \bar x\|^2|\F_k] \le \|x_k - \bar x\|^2 - c_0 \|A x_k -y\|^2
\end{align}
for all integers $k \ge 0$, where $c_0:= (2- \gamma \|A\|^2) \gamma/b > 0$. 
%Moreover
%\begin{align}\label{RBCD.3}
%\EE[\|A x_k - y\|^2] \le \frac{\|x_0-\bar x\|^2}{c_0 (k+1)}, \quad \forall k \ge 0. 
%\end{align}
\end{lemma}

\begin{proof}
Let $\bar x_i$ denote the $i$-th component of $\bar x$, i.e. $\bar x = (\bar x_1, \cdots, \bar x_b)$. 
By the definition of $x_{k+1}$ and the polarization identity, we have 
\begin{align*}
\|x_{k+1} - \bar x\|^2 
& = \sum_{i\ne i_k} \|x_{k+1, i} - \bar x_i\|^2 + \|x_{k+1, i_k} - \bar x_{i_k}\|^2  \displaybreak[0]\\
& = \sum_{i\ne i_k} \|x_{k,i} - \bar x_i\|^2 + \|x_{k,i_k} - \bar x_{i_k} - \gamma A_{i_k}^*(A x_k - y)\|^2  \displaybreak[0]\\
& = \sum_{i\ne i_k} \|x_{k,i} - \bar x_i\|^2 + \|x_{k,i_k} - \bar x_{i_k}\|^2 
+ \gamma^2 \|A_{i_k}^*(A x_k - y)\|^2  \\
& \quad \, - 2 \gamma \l x_{k,i_k} - \bar x_{i_k}, A_{i_k}^*(A x_k - y)\r \\
& = \|x_k - \bar x\|^2 + \gamma^2 \|A_{i_k}^*(A x_k - y)\|^2  \displaybreak[0]\\
& \quad \, - 2\gamma \l A_{i_k}(x_{k, i_k} - \bar x_{i_k}), A x_k - y\r.
\end{align*}
Taking the conditional expectation and using $A \bar x = y$, we can obtain 
\begin{align*}
& \EE[\|x_{k+1} - \bar x\|^2|\F_k] \\
& = \|x_k - \bar x\|^2 + \frac{\gamma^2}{b} \sum_{i=1}^b \|A_i^* (A x_k - y)\|^2 
- \frac{2 \gamma}{b} \sum_{i=1}^b \l A_i (x_{k,i} - \bar x_i), A x_k - y\r \displaybreak[0]\\
& = \|x_k - \bar x\|^2 + \frac{\gamma^2}{b}\|A^* (A x_k - y)\|^2  
- \frac{2\gamma}{b} \l A(x_k - \bar x), A x_k - y\r \displaybreak[0]\\
& \le \|x_k - \bar x\|^2 + \frac{\gamma^2 \|A\|^2}{b} \|A x_k - y\|^2 
- \frac{2\gamma}{b} \|A x_k - y\|^2
\end{align*}
which shows (\ref{RBCD.1}). 
%
%By taking the full expectation and using the tower property, we have from (\ref{RBCD.1}) that 
%$$
%\EE[\|x_{k+1} - \bar x\|^2] \le \EE[\|x_k - \bar x\|^2] - c_0 \EE[\|A x_k - y\|^2]
%$$
%which implies 
%\begin{align}\label{RBCD.2}
%c_0 \sum_{l=0}^k \EE[\|A x_l - y\|^2] \le \EE[\|x_0-\bar x\|^2] - \EE[\|x_{k+1} - \bar x\|^2]
%\le \|x_0 - \bar x\|^2. 
%\end{align}
%To proceed further, we next show the monotonicity of $\{\|A x_k - y\|^2\}$. We have 
%\begin{align*}
%\|A x_{k+1} - y\|^2 
%& = \|A x_k - y\|^2 + \|A (x_{k+1} - x_k)\|^2 + 2\l A(x_{k+1} - x_k), A x_k - y\r \\
%& = \|A x_k - y\|^2 + \gamma^2\|A_{i_k} A_{i_k}^*(A x_k - y)\|^2 \\
%& \quad \, -2\gamma\l A_{i_k} A_{i_k}^*(A x_k -y), A x_k - y\r \\
%& =  \|A x_k - y\|^2 + \gamma^2\|A_{i_k} A_{i_k}^*(A x_k - y)\|^2 
%-2\gamma \|A_{i_k}^*(A x_k -y)\|^2 \\
%& \le  \|A x_k - y\|^2 - (2 - \gamma \|A_{i_k}\|^2) \gamma \|A_{i_k}^*(A x_k - y)\|^2.
%\end{align*}
%Note that $\|A_{i_k}\| \le \|A\|$. Since $0< \gamma < 2/\|A\|^2$, we therefore obtain 
%$\|A x_{k+1} - y\|^2 \le \|A x_k -y\|^2$ for all $k \ge 0$. Using this monotonicity in 
%(\ref{RBCD.2}) we can obtain 
%\begin{align*}
%c_0 (k+1) \EE[\|A x_k - y\|^2] \le \|x_0 - \bar x\|^2. 
%\end{align*}
%which shows (\ref{RBCD.3}). The proof is thus complete. 
\end{proof}

To proceed further, we need the following Doob's martingale convergence theorem (\cite{B2020}). 

\begin{proposition}\label{prop:DMT}
Let $\{U_k\}$ be a sequence of nonnegative random variables in a probability space that is a 
supermartingale with respect to a filtration $\{\F_k\}$, i.e. 
$$
\EE[U_{k+1}|\F_k] \le U_k, \quad \forall k.  
$$
Then $\{U_k\}$ converges almost surely to a nonnegative random variable $U$ with finite 
expectation.
\end{proposition}

Based on Proposition \ref{prop:DMT}, we next prove the almost sure weak convergence of $\{x_k\}$. 
We need $\X$ to be separable in the sense that $\X$ has a countable dense subset. 

\begin{theorem}\label{RBCD.thm1}
Consider the sequence $\{x_k\}$ defined by Algorithm \ref{alg:RBCD} using exact data. Assume that 
$\X$ is separable and $0< \gamma <2/\|A\|^2$, then $\{x_k\}$ converges weakly to a random solution 
$\bar x$ of (\ref{eq:problem}) almost surely. 
\end{theorem}

\begin{proof}
Let $S$ denote the set of solutions of (\ref{eq:problem}). According to Lemma \ref{RBCD.lem2}, 
we have for any solution $z \in S$ that 
$$
\EE[\|x_{k+1} - z\|^2|\F_k] \le \|x_k - z\|^2, \quad \forall k
$$
which means $\{\|x_k - z\|^2\}$ is a supermartingale. Thus, we may use Proposition \ref{prop:DMT}
to conclude that the event 
$$
\Omega_z:= \left\{\lim_{k\to \infty} \|x_k - z\| \mbox{ exists  and is finite}\right\}
$$
has probability one. We now strengthen this result by showing that there is an event $\Omega_1$ 
of probability one such that, for any $\tilde z \in S$ and any sample path $\omega \in \Omega_1$, 
the limit $\lim_{k \to \infty} \|x_k(\omega) - \tilde z\|$ exists. We adapt the arguments 
in \cite{B2011,CP2015}. By the separability of $\X$, we can find a countable set $C\subset S$ 
such that $C$ is dense in $S$. Let 
$$
\Omega_1 := \bigcap_{z\in C} \Omega_z.
$$
Since $\P(\Omega_z) = 1$ for each $z \in C$ and $C$ is countable, we have $\P(\Omega_1) = 1$. 
Let $\tilde z\in S$ be any point. Then there is a sequence $\{z_l\}\subset C$ such that 
$z_l \to \tilde z$ as $l \to \infty$. For any sample path $\omega \in \Omega_1$ we have by the 
triangle inequality that 
$$
-\|z_l - \tilde z\| \le \|x_k(\omega) - \tilde z\| - \|x_k(\omega) - z_l\| \le \|z_l - \tilde z\|. 
$$
Thus 
\begin{align*}
-\|z_l - \tilde z\| & \le \liminf_{k \to \infty} \left\{\|x_k(\omega)-\tilde z\| 
- \|x_k(\omega)-z_l\|\right\} \displaybreak[0]\\
& \le \limsup_{k \to \infty} \left\{\|x_k(\omega)-\tilde z\| - \|x_k(\omega)-z_l\|\right\} \displaybreak[0] \\
& \le \|z_l - \tilde z\|. 
\end{align*}
Since $\omega \in \Omega_1 \subset \Omega_{z_l}$, $\lim_{k\to \infty} \|x_k(\omega) - z_l\|$ exists. 
Thus, by the properties of $\liminf$ and $\limsup$ we have 
\begin{align*}
-\|z_l - \tilde z\|
& \le \liminf_{k \to \infty} \|x_k(\omega)-\tilde z\| - \lim_{k\to \infty} \|x_k(\omega)-z_l\| \displaybreak[0]\\
& \le \limsup_{k \to \infty} \|x_k(\omega)-\tilde z\| - \lim_{k\to \infty}\|x_k(\omega)-z_l\| \displaybreak[0]\\
& \le \|z_l - \tilde z\|. 
\end{align*}
This implies that both $\liminf_{k\to \infty} \|x_k(\omega) - \tilde z\|$ and 
$\limsup_{k\to \infty} \|x_k(\omega) - \tilde z\|$ are finite with 
$$
0 \le \limsup_{k\to \infty} \|x_k(\omega) - \tilde z\| - \liminf_{k\to \infty} \|x_k(\omega) - \tilde z\|
\le 2 \|z_l - \tilde z\|.
$$
Letting $l \to \infty$ shows that 
$$
\liminf_{k\to \infty} \|x_k(\omega) - \tilde z\| 
= \limsup_{k\to \infty} \|x_k(\omega) - \tilde z\|
$$
and hence $\lim_{k \to \infty} \|x_k(\omega) - \tilde z\|$ exists and is finite for every 
$\omega \in \Omega_1$ and $\tilde z \in S$. 

Next we use Lemma \ref{RBCD.lem2} again to obtain for any $z \in S$ that 
$$
\EE[\|x_{k+1} - z\|^2] \le \EE[\|x_k - z\|^2] - c_0 \EE[\|A x_k - y\|^2] 
$$
which implies that 
$$
\EE\left[\sum_{k=0}^\infty \|A x_k - y\|^2 \right] 
= \sum_{k=0}^\infty \EE[\|A x_k - y\|^2] \le \frac{\|x_0 - z\|^2}{c_0} <\infty.
$$
Consequently, the event  
\begin{align}\label{Omega2}
\Omega_2:= \left\{ \sum_{k=0}^\infty \|A x_k - y\|^2 < \infty\right\}
\end{align}
has probability one. Let $\Omega_0:= \Omega_1 \cap \Omega_2$. Then $\P(\Omega_0) = 1$. 
Note that along any sample path in $\Omega_0$ we have $\{\|x_k - z\|\}$ is convergent for any 
$z \in S$ and 
$$
\sum_{k=0}^\infty \|A x_k - y\|^2 < \infty
$$
which implies $\|A x_k - y\| \to 0$ as $k \to \infty$. The convergence of $\{\|x_k - z\|\}$
implies that $\{x_k\}$ is bounded and hence it has a weakly convergent subsequence $\{x_{k_l}\}$ 
such that $x_{k_l} \rightharpoonup \bar x$ as $l \to \infty$ for some $\bar x \in \X$, where 
``$\rightharpoonup$" denotes the weak convergence. Since $A x_k \to y$, we must have $A \bar x = y$, 
i.e. $\bar x \in S$ and consequently $\|x_k - \bar x\|$ converges. We now show that 
$x_k \rightharpoonup \bar x$ for the whole sequence $\{x_k\}$. It suffices to show that $\bar x$ is the unique weak cluster 
point of $\{x_k\}$. Let $x^*$ be any cluster point of $\{x_k\}$. Then there is another 
subsequence $\{x_{n_l}\}$ of $\{x_k\}$ such that $x_{n_l} \rightharpoonup x^*$. From the above 
argument we also have $x^* \in S$ and thus $\{\|x_k - x^*\|\}$ is convergent. Noting the 
identity
$$
\l x_k, x^* - \bar x\r = \|x_k - \bar x\|^2 - \|x_k - x^*\|^2 - \|\bar x\|^2 + \|x^*\|^2. 
$$
Since both $\{\|x_k-\bar x\|\}$ and $\{\|x_k - x^*\|\}$ are convergent, we can conclude that 
$\lim_{k\to \infty} \l x_k, x^*- \bar x\r$ exists. Therefore 
\begin{align*}
\lim_{k\to \infty} \l x_k, x^* - \bar x\r & = \lim_{l\to \infty} \l x_{k_l}, x^*- \bar x\r 
= \l \bar x, x^* - \bar x\r, \\
\lim_{k\to \infty} \l x_k, x^* - \bar x\r & = \lim_{l\to \infty} \l x_{n_l}, x^* - \bar x\r 
= \l x^*, x^* - \bar x\r
\end{align*}
and thus $\l \bar x, x^* - \bar x\r = \l x^*, x^* - \bar x\r$, i.e. $\|x^* - \bar x\|^2 =0$
and hence $x^* = \bar x$. The proof is complete. 
\end{proof}

\begin{remark}\label{RBCD.rem1}
{\rm 
Theorem \ref{RBCD.thm1} shows that there is an event $\Omega_0$ of probability one and a 
random vector $\bar x$ such that $A \bar x = y$ almost surely and $x_k \rightharpoonup \bar x$ 
along every sample path in $\Omega_0$. Let $x^\dag$ denote the unique $x_0$-minimum-norm 
solution of $A x = y$. We would like to point out that $\bar x = x^\dag$ almost surely 
if $\{x_k\}$ is uniformly bounded in the sense that 
\begin{align}\label{bound}
\mbox{there is a constant } C \mbox{ such that } \|x_k\| \le C \mbox{ for all } k \mbox{ almost surely}. 
\end{align}
To see this, we first claim that 
\begin{align}\label{claim}
\EE[\l x_k - x_0, \bar x - x^\dag\r] = 0, \quad \forall k.  
\end{align}
Indeed this is trivial for $k = 0$. Assume it is true for some $k \ge 0$. Then, by the definition of $x_{k+1}$, we have 
\begin{align*}
\l x_{k+1} - x_0, \bar x - x^\dag\r 
& = \l x_k - x_0, \bar x - x^\dag\r - \gamma \l A_{i_k}^*(A x_k - y), \bar x_{i_k} - x_{i_k}^\dag\r \\
& = \l x_k - x_0, \bar x - x^\dag\r - \gamma \l A x_k - y, A_{i_k} (\bar x_{i_k} - x_{i_k}^\dag)\r. 
\end{align*}
Thus
\begin{align*}
\EE[\l x_{k+1} - x_0, \bar x - x^\dag\r|\F_k] 
& = \l x_k - x_0, \bar x - x^\dag\r - \frac{\gamma}{b} \left\l A x_k - y, \sum_{i=1}^b A_i (\bar x_i - x_i^\dag)\right\r \\
& = \l x_k - x_0, \bar x - x^\dag\r - \frac{\gamma}{b} \l A x_k - y, A(\bar x - x^\dag)\r \\
& = \l x_k - x_0, \bar x - x^\dag\r
\end{align*}
because $A \bar x = y = A x^\dag$ almost surely. Consequently, by taking the full expectation and using the induction hypothesis we can obtain $\EE[\l x_{k+1} - x_0, \bar x - x^\dag\r] = 0$ and the claim (\ref{claim}) is proved. 

Under the condition (\ref{bound}), we have 
$$
|\l x_k - x_0, \bar x - x^\dag\r| \le \|x_k-x_0\| \|\bar x - x^\dag\|\le (C+\|x_0\|)\|\bar x - x^\dag\|.
$$
Since $x_k \rightharpoonup \bar x$ almost surely, by using the weak lower semi-continuity of norms and the Fatou lemma we can obtain from Lemma \ref{RBCD.lem2} that 
\begin{align*}
\EE[\|\bar x - x^\dag\|^2] 
& \le \EE\left[\liminf_{k\to \infty} \|x_k-x^\dag\|^2\right] 
\le \liminf_{k\to \infty} \EE[\|x_k - x^\dag\|^2] \\
& \le \|x_0 - x^\dag\|^2 <\infty
\end{align*}
and thus 
$$
\EE[\|\bar x - x^\dag\|] \le \left(\EE[\|\bar x - x^\dag\|^2]\right)^{1/2} \le \|x_0 - x^\dag\| <\infty. 
$$
Therefore, we may use the dominated convergence theorem, $\l x_k, \bar x - x^\dag\r \to \l \bar x, \bar x- x^\dag\r$ and (\ref{claim}) to conclude 
$$
\EE[\l \bar x-x_0, \bar x - x^\dag\r] = \lim_{k\to \infty} \EE[\l x_k - x_0, \bar x - x^\dag\r] = 0. 
$$
Since $\l x^\dag-x_0, \bar x - x^\dag\r =0$ always holds as $x^\dag$ is the $x_0$-minimum-norm 
solution, we thus obtain $\EE[\|\bar x - x^\dag\|^2] =0$ which implies that $\bar x = x^\dag$ 
almost surely. 

It should be emphasized that the above characterization of $\bar x$ depends crucially 
on the condition (\ref{bound}). We haven't yet verified it for general forward operator $A$. 
However, for the operator $A$ with a particular tensor product form as studied in \cite{RNH2019}, 
we will show that (\ref{bound}) holds in subsection \ref{subsect2.4}. 
}
\end{remark}

By using Lemma \ref{RBCD.lem1} and Theorem \ref{RBCD.thm1}, now we are ready to prove the 
following weak convergence result on Algorithm \ref{alg:RBCD} under an {\it a priori} 
stopping rule. 

\begin{theorem}\label{RBCD.thm2}
For any sequence $\{y^{\d_n}\}$ of noisy data satisfying $\|y^{\d_n}-y\|\le \d_n$ with 
$\d_n \to 0$ as $n \to \infty$, let $\{x_k^{\d_n}\}$ be the iterative sequence produced 
by Algorithm \ref{alg:RBCD} with $y^\d$ replaced by $y^{\d_n}$, where $0< \gamma < 2/\|A\|^2$. 
Let the integer $k_n$ be chosen such that $k_n \to \infty$ and $k_n \d_n^2 \to 0$ as 
$n\to \infty$. Then, by taking a subsequence of $\{x_{k_n}^{\d_n}\}$ if necessary, there 
holds $x_{k_n}^{\d_n} \rightharpoonup \bar x$ as $n \to \infty$ almost surely, where $\bar x$ 
denotes the random solution of $A x = y$ determined in Theorem \ref{RBCD.thm1}. 
\end{theorem}

\begin{proof}
According to Lemma \ref{RBCD.lem1} and $k_n \d_n^2 \to 0$ we have 
$$
\EE\left[\|x_{k_n}^{\d_n} - x_{k_n}\|^2\right] \le \frac{2\gamma}{b} k_n \d_n^2 \to 0 
\quad \mbox{as } n\to \infty. 
$$
Therefore, by taking a subsequence of $\{x_{k_n}^{\d_n}\}$ if necessary, we can find an event 
$\Omega_3$ with ${\mathbb P}(\Omega_3) =1$ such that $x_{k_n}^{\d_n} - x_{k_n} \to 0$ along every 
sample path in $\Omega_3$. According to Theorem \ref{RBCD.thm1} and $k_n \to \infty$, there is 
an event $\Omega_4$ of probability one such that $x_{k_n} \rightharpoonup \bar x$ as $n \to \infty$ 
along every sample path in $\Omega_4$. Let $\Omega := \Omega_3\cap \Omega_4$. Then 
${\mathbb P}(\Omega) =1$ and for any $x \in \X$ there holds 
\begin{align*}
\l x_{k_n}^{\d_n} - \bar x, x\r 
= \l x_{k_n}^{\d_n} - x_{k_n}, x\r + \l x_{k_n} - \bar x, x\r \to 0
\end{align*}
as $n\to \infty$ along every sample path in $\Omega$. The proof is thus complete. 
\end{proof}

\subsection{\bf The discrepancy principle}

The above weak convergence result on Algorithm \ref{alg:RBCD} is established under an {\it a priori} 
stopping rule. In applications, we usually expect to terminate the iteration by {\it a posteriori} rules. 
Note that $r_k^\d:=A x_k^\d - y^\d$ is involved in the algorithm, it is natural to consider terminating 
the iteration by the discrepancy principle which determines $k_\d$ to be the first integer such that 
$\|r_k^\d\| \le \tau \d$, where $\tau>1$ is a given number. Incorporating the discrepancy principle 
into Algorithm \ref{alg:RBCD} leads to the following algorithm. 

\begin{algorithm}[RBCD method with the discrepancy principle]\label{alg:RBCD_DP} 
Pick an initial guess $x_0 \in \X$, set $x_0^\d := x_0$ and calculate $r_0^\d:= A x_0^\d - y^\d$. 
Choose $\tau>1$ and $\gamma>0$. For all integers $k \ge 0$ 
do the following:  
\begin{enumerate}

\item[(i)] Set the step size $\gamma_k$ by 
\begin{align*}
\gamma_k := \left\{\begin{array}{lll}
\gamma & \mbox{ if } \|r_k^\d\| > \tau \d,\\
0 & \mbox{ if } \|r_k^\d\| \le \tau \d;
\end{array}\right.
\end{align*}

\item[(ii)] Pick an index $i_k \in \{1, \cdots, b\}$ randomly via the uniform distribution; 

\item[(iii)] Update $x_{k+1}^\d$ by setting $x_{k+1,i}^\d = x_{k,i}^\d$ for $i \ne i_k$ and 
$$
x_{k+1, i_k}^\d = x_{k, i_k}^\d - \gamma_k A_{i_k}^* r_k^\d;
$$

\item[(iv)] Calculate $r_{k+1}^\d = r_k^\d + A_{i_k}(x_{k+1, i_k}^\d - x_{k,i_k}^\d)$. 
\end{enumerate}
\end{algorithm}

Algorithm \ref{alg:RBCD_DP} is formulated in the way that it incorporates the discrepancy principle to define an infinite sequence $\{x_k^\d\}$, which is convenient for the analysis below. In numerical simulations, the iteration should be terminated as long as $\|r_k^\d\| \le \tau \d$ is satisfied because the iterates are no longer updated. It should be highlighted that the stopping index depends crucially on the sample path and thus is a random integer. Note also that the step size $\gamma_k$ in Algorithm \ref{alg:RBCD_DP} is a random number; this sharply contrasts to the step size $\gamma$ in Algorithm \ref{alg:RBCD} which is deterministic. 

\begin{proposition}\label{RBCD.DP:prop1}
Consider Algorithm \ref{alg:RBCD_DP} with $\gamma = \mu/\|A\|^2$ for some $0<\mu<2$.  
Then the iteration must terminate in finite many steps almost surely. If in addition
$0<\mu <2-2/\tau$, then for any solution $\bar x$ of (\ref{eq:problem}) there holds 
\begin{align}\label{RBCD.DP1}
\EE[\|x_{k+1}^\d - \bar x\|^2] \le \EE[\|x_k^\d - \bar x\|^2] 
- c_1 \EE\left[\gamma_k \|A x_k^\d - y^\d\|^2\right] 
\end{align}
for all integers $k\ge 0$, where $c_1: = (2-2/\tau - \mu)/b > 0$.  
\end{proposition}

\begin{proof}
By virtue of (\ref{RBCD.-2}) in Lemma \ref{RBCD.lem1} we have along any sample path that 
$$
\|A x_k^\d - y^\d\| \le \|A x_k - y\| + \|A (x_k^\d - x_k) - y^\d + y\| 
\le \|A x_k - y\| + \d. 
$$
In the proof of Theorem \ref{RBCD.thm1} we have shown that 
$$
\P\left(\lim_{k\to \infty} \|A x_k - y\|=0\right) = 1.
$$
Therefore, as $\tau>1$, it follows that  
\begin{align*}
\limsup_{k\to \infty} \|A x_k ^\d - y^\d\| \le \d <\tau \d \quad \mbox{almost surely}
\end{align*}
which means that $\|A x_k^\d - y^\d\| <\tau \d$ for some finite integer $k$ almost surely, 
i.e. Algorithm \ref{alg:RBCD_DP} must terminate in finite many steps almost surely. 

Next we show (\ref{RBCD.DP1}) under the additional condition $0 < \mu < 2-2/\tau$. 
By following the proof of Lemma \ref{RBCD.lem2} we can obtain 
\begin{align*}
\|x_{k+1}^\d - \bar x\|^2 
& = \|x_k^\d - \bar x\|^2 + \gamma_k^2 \|A_{i_k}^* (A x_k^\d - y^\d)\|^2 \\
& \quad \, - 2\gamma_k \l A_{i_k}(x_{k, i_k}^\d - \bar x_{i_k}), A x_k^\d - y^\d \r. 
\end{align*}
By taking the conditional expectation on $\F_k$ and noting that $\gamma_k$ is $\F_k$-measurable, 
we have 
\begin{align*}
&\EE[\|x_{k+1}^\d - \bar x\|^2|\F_k] \\
& = \|x_k^\d - \bar x\|^2 + \frac{\gamma_k^2}{b} \sum_{i=1}^b \|A_i^*(A x_k^\d - y^\d)\|^2
- \frac{2\gamma_k}{b} \left\l \sum_{i=1}^b A_i (x_{k,i}^\d - \bar x_i), A x_k^\d - y^\d\right\r \displaybreak[0]\\
& = \|x_k^\d - \bar x\|^2 + \frac{\gamma_k^2}{b} \|A^*(A x_k^\d - y^\d)\|^2 
- \frac{2\gamma_k}{b} \l A(x_k^\d - \bar x), A x_k^\d - y^\d\r \displaybreak[0] \\
& \le \|x_k^\d - \bar x\|^2 + \frac{\gamma_k^2\|A\|^2}{b} \|A x_k^\d - y^\d\|^2  
- \frac{2\gamma_k}{b} \|A x_k^\d - y^\d\|^2 + \frac{2\gamma_k \d}{b} \|A x_k^\d - y^\d\|.
\end{align*}
By the definition of $\gamma_k$ we have $\gamma_k \d \le \frac{\gamma_k}{\tau} \|A x_k^\d - y^\d\|$. 
Therefore 
\begin{align*}
\EE[\|x_{k+1}^\d - \bar x\|^2|\F_k] 
& \le \|x_k^\d - \bar x\|^2 - \frac{1}{b}\left(2- \frac{2}{\tau} -\gamma_k \|A\|^2\right) \gamma_k \|A x_k^\d - y^\d\|^2 \\
& \le \|x_k^\d - \bar x\|^2 - \frac{1}{b} \left(2- \frac{2}{\tau} -\mu\right) \gamma_k \|A x_k^\d - y^\d\|^2.
\end{align*}
Taking the full expectation gives (\ref{RBCD.DP1}). 
\end{proof}

\subsection{\bf Strong convergence}\label{subsect2.4}

In Subsection \ref{sect2.2}, we obtained weak convergence result on Algorithm \ref{alg:RBCD}. In this 
section, we will show that, for a special case studied in \cite{RNH2019}, a strong convergence result can be 
derived. To start with, let $V=(v_{ij})$ be a $d\times b$ matrix and let $K: X \to Y$ be a bounded linear 
operator. We will consider the problem (\ref{prob0}) which can be written as $Ax=y$ by setting 
$y:= (y_1, \cdots, y_d) \in Y^d$ and define $A: X^b \to Y^d$ as 
$$
Ax:= \left( \sum_{i=1}^b v_{li} K x_i \right)_{l=1}^d, \quad \forall x=(x_i)_{i=1}^b \in X^b.
$$
It is easy to see that (\ref{prob0}) is a special case of (\ref{eq:problem}) with $X_i = X$ for each $i$, 
$\Y=Y^d$, and 
$$
A_i z := \left(v_{li} Kz\right)_{l=1}^d, \quad \forall z \in X.
$$
Note that $A^*_i:Y^d \to X$, the adjoint of $A_i$, has the following form
$$
A^*_i \tilde y = \sum_{l=1}^d v_{li} K^* \tilde y_l, \quad \forall \tilde y=(\tilde y_l)_{l=1}^d\in Y^d. 
$$
Thus, when our randomized block coordinate descent method, i.e. Algorithm \ref{alg:RBCD}, is used to solve 
(\ref{prob0}) with the exact data, the iteration scheme becomes $x_{k+1,i}=x_{k,i}$ if $i \neq i_k$ and
\begin{align*}
x_{k+1,i_k} = x_{k,i_k} - \gamma A_{i_k}^* (A x_k - y) 
= x_{k,i_k} - \gamma \sum_{l=1}^d v_{l i_k} K^* (Ax_k - y )_l, 
\end{align*}
where $(Ax_k - y )_l$ denotes the $l$-th component of $Ax_k - y$. In order to give a convergence analysis, 
we introduce $\theta_{k} = (\theta_{k,l})_{l=1}^d$ with 
$$
\theta_{k,l} = \sum_{i=1}^b v_{li} x_{k,i}, \quad l=1,\cdots,d
$$
and for any solution $\hat x$ of (\ref{prob0}) we set $\hat \theta = (\hat \theta_l)_{l=1}^d$ with
$$
\hat \theta_l = \sum_{i=1}^b v_{li} \hat x_i,  \quad l = 1,\cdots,d.
$$
Then, we have
\begin{align*}
\theta_{k+1,l} 
&= \sum_{i\neq i_k} v_{li} x_{k+1,i} + v_{l i_k} x_{k+1,i_k}\\
&= \sum_{i\neq i_k} v_{li } x_{k,i} + v_{l i_k} x_{k, i_k} 
- \gamma v_{l i_k} \sum_{l'=1}^d v_{l' i_k} K^* (Ax_k -y)_{l'}\\
&= \theta_{k,l} - \gamma v_{l i_k} \sum_{l'=1}^d v_{l' i_k} K^* (A x_k -y)_{l'}.
\end{align*}
Consequently, 
\begin{align}
\| \theta_{k+1} - \hat \theta \|^2 
&= \sum_{l=1}^d \|\theta_{k+1,l} - \hat\theta_l \|^2 
= \sum_{l=1}^d \left\|\theta_{k,l} - \hat \theta_l 
- \gamma v_{l i_k} \sum_{l'=1}^d v_{l' i_k} K^* (A x_k -y)_{l'}\right\|^2 \notag\\
&= \sum_{l=1}^d \|  \theta_{k,l} - \hat \theta_l \|^2 - 2 \gamma \Delta_1 
+ \gamma^2 \Delta_2, \label{RBCD4special.1}       
\end{align}
where
\begin{align*}
\Delta_1 &= \sum_{l=1}^d \left\l \theta_{k,l} - \hat \theta_l, 
v_{l i_k} \sum_{l'=1}^d v_{l' i_k} K^* (Ax_k - y)_{l'} \right\r, \\
\Delta_2 &=  \sum_{l=1}^d \left\| v_{l i_k} \sum_{l'=1}^d v_{l' i_k} K^* (Ax_k - y)_{l'} \right\|^2.
\end{align*} 
By straightforward calculation, we can obtain 
\begin{align*}
\Delta_1 
&= \sum_{l=1}^d \sum_{i=1}^b v_{li} \left\l x_{k,i} - \hat x_i, v_{l i_k} 
\sum_{l'=1}^d v_{l' i_k} K^* (Ax_k - y)_{l'} \right\r \displaybreak[0]\\
&= \sum_{l=1}^d \sum_{l'=1}^d v_{l i_k} v_{l' i_k} 
\left< \sum_{i=1}^b v_{li} K(x_{k,i} - \hat x_i), (Ax_k - y)_{l'} \right> \displaybreak[0] \\
&= \sum_{l=1}^d \sum_{l'=1}^d v_{l i_k} v_{l' i_k} \left< (A x_k -y)_l, (A x_k - y)_{l'} \right> \\
&= \left\| \sum_{l=1}^d v_{l i_k} (A x_k - y)_l \right\|^2
\end{align*}
and 
\begin{align*}
\Delta_2 
&\leq \sum_{l=1}^d | v_{l i_k} |^2 \| K \|^2 \left\| \sum_{l'=1}^d v_{l' i_k} (Ax_k - y)_{l'} \right\|^2
= \| v_{i_k}\|^2 \| K \|^2 \left\|\sum_{l=1}^d v_{l i_k} (Ax_k - y)_l\right\|^2,
\end{align*}
where, for each $i$, we use $v_i$ to denote the $i$-th column of $V$. Combining these two equations 
with (\ref{RBCD4special.1}) gives
$$
\| \theta_{k+1} - \hat \theta \|^2 
\le \| \theta_{k} - \hat \theta \|^2 
- (2- \gamma\|v_{i_k}\|^2 \|K\|^2)\gamma \left\|\sum_{l=1}^d v_{l i_k} (Ax_k - y)_l\right\|^2
$$
which shows the following result.

\begin{lemma}\label{RBCD4special.lem1}
Consider Algorithm \ref{alg:RBCD} for solving (\ref{prob0}) with the exact data. Let 
$v^* = \max\{\|v_i\|^2: i = 1, \cdots, b\}$. If $0< \gamma < 2/(v^* \|K\|^2)$, then
$$
\| \theta_{k+1} - \hat \theta \|^2 \leq \| \theta_{k} - \hat \theta \|^2 
- c_2 \left\|\sum_{l=1}^d v_{l i_k} (Ax_k - y)_l \right\|^2
$$
for all $k\geq 0$, where $c_2 := (2-\gamma v^* \|K\|^2)\gamma >0$. Consequently, 
$\{\|\theta_k - \hat \theta\|^2\}$ is monotonically decreasing. 
\end{lemma}

Based on Lemma \ref{RBCD4special.lem1}, now we show the almost sure strong convergence of $\{x_k\}$ 
under the assumption that $V=(v_{l i})$ is of full column rank.

\begin{theorem}\label{RBCD.thm11}
Consider Algorithm \ref{alg:RBCD} for solving (\ref{prob0}) with the exact data. Assume that 
$V=(v_{l i})$ is of full column rank and let $v^*$ be defined as in Lemma \ref{RBCD4special.lem1}. 
If $0<\gamma < 2/(v^* \|K\|^2)$, then $\|x_k -x^\dag\|\to 0$ almost surely and 
$\EE[\|x_k - x^\dag\|^2] \to 0$ as $k \to \infty$, where $x^\dag$ denotes the unique 
$x_0$-minimum-norm solution of (\ref{prob0}).
\end{theorem}

\begin{proof}
Consider the event $\Omega_2$ defined in (\ref{Omega2}). It is known that $\P(\Omega_2) =1$. 
We now fix an arbitrary sample path $\{i_k: k= 0, 1, \cdots\}$ in $\Omega_2$ and show that, 
along this sample path, $\{\theta_k\}$ is a Cauchy sequence. Recall that 
\begin{align}\label{RBCD.32}
\sum_{k=0}^\infty \|A x_k - y\|^2 < \infty 
\end{align}
and hence $\|A x_k - y\| \to 0$ as $k\to \infty$. Given any two positive integers $p \leq q$, let 
$k^*$ be an integer such that $p\leq k^* \leq q$ and 
\begin{align}\label{RBCD.31}
\|A x_{k^*} - y\| = \min \left\{\|A x_k - y\|: p\le k \le q\right\}. 
\end{align}
Then 
$$
\| \theta_p - \theta_q \|^2 
\le 2 \left(\|\theta_p - \theta_{k^*} \|^2 + \| \theta_q - \theta_{k^*} \|^2 \right).
$$
We now show that $\| \theta_p - \theta_{k^*} \|^2 \rightarrow 0$ and 
$\| \theta_q - \theta_{k^*} \|^2 \rightarrow 0$ as $p \to \infty$. To show the 
first assertion, we note that
$$
\| \theta_p - \theta_{k^*} \|^2 = \| \theta_p - \hat \theta \|^2 - \| \theta_{k^*} - \hat \theta \|^2 
+ 2 \left< \theta_{k^*} - \theta_p,  \theta_{k^*} - \hat \theta \right>.
$$
Since, by Lemma \ref{RBCD4special.lem1}, $\{\|\theta_p - \hat \theta\|\}$ is 
monotonically decreasing, $\lim_{p\to \infty} \|\theta_p - \hat \theta\|$ exists and thus 
$\| \theta_p - \hat \theta \|^2 - \| \theta_{k^*} - \hat \theta \|^2 \to 0$ as $p \to \infty$. 
To estimate $\l\theta_{k^*} - \theta_p,  \theta_{k^*} - \hat \theta \r$, we first write 
\begin{align*}
 \left\l \theta_{k^*} - \theta_p,  \theta_{k^*} - \hat \theta \right\r
 = \sum_{k=p}^{k^* - 1} \left\l \theta_{k+1} - \theta_k, \theta_{k^*} - \hat \theta \right\r.
\end{align*}
By the definition of $\theta_{k+1}$ and $\hat \theta$, we have
\begin{align*}
\left\l \theta_{k^*} - \theta_p,  \theta_{k^*} - \hat \theta \right\r 
& = \sum_{k=p}^{k^*-1} \sum_{l=1}^d \left\l \theta_{k+1,l} - \theta_{k,l}, \theta_{k^*,l} - \hat \theta_l \right\r \\
&= \sum_{k=p}^{k^*-1} \sum_{l=1}^d \left\l \sum_{i=1}^b v_{li} (x_{k+1,i} - x_{k,i}), 
\sum_{i'=1}^b v_{l i'} (x_{k^*,i'} - \hat x_{i'}) \right\r \displaybreak[0] \\
&= \sum_{k=p}^{k^*-1} \sum_{l=1}^d \left< v_{l i_k} (x_{k+1,i_k} - x_{k,i_k}), 
\sum_{i' =1}^b v_{l i'} (x_{k^*,i'} - \hat x_{i'}) \right> \displaybreak[0] \\
&= -\gamma \sum_{k=p}^{k^*-1} \sum_{l=1}^d \left< v_{l i_k} A_{i_k}^*(Ax_k - y), 
\sum_{i'=1}^b v_{l i'} (x_{k^*,i'} - \hat x_{i'}) \right>  \displaybreak[0] \\
&= - \gamma \sum_{k=p}^{k^* - 1} \sum_{l=1}^d \sum_{i'=1}^b v_{l i_k} v_{l i'}
\left\l Ax_k - y, A_{i_k} (x_{k^*,i'} - \hat x_{i'}) \right\r.
\end{align*}
Using the definition of $A_{i_k}$ we further have 
\begin{align*}
\left\l \theta_{k^*} - \theta_p,  \theta_{k^*} - \hat \theta \right\r
&= - \gamma \sum_{k=p}^{k^* - 1} \sum_{l=1}^d \sum_{i'=1}^b v_{l i_k} v_{l i'} \left< Ax_k - y,  
\left( v_{l' i_k} K(x_{k^*,i'} - \hat x_{i'})\right)_{l'=1}^d \right> \displaybreak[0] \\
&= - \gamma  \sum_{k=p}^{k^* - 1} \sum_{l,l'=1}^d \sum_{i' =1}^b v_{l i_k} v_{l' i_k} v_{l i'} 
\left< (Ax_k - y)_{l'},   K(x_{k^*,i'} - \hat x_{i'})\right> \displaybreak[0] \\
&= - \gamma \sum_{k=p}^{k^*-1} \sum_{l,l'=1}^d v_{l i_k} v_{l' i_k} 
\left< (Ax_k - y)_{l'}, (Ax_{k^*} - y)_{l} \right> \displaybreak[0] \\
&= - \gamma \sum_{k=p}^{k^*-1} \left< \sum_{l'=1}^d v_{l' i_k}(Ax_k - y)_{l'}, 
\sum_{l=1}^d v_{l i_k} (Ax_{k^*} - y)_l \right>.
\end{align*}
By virtue of the Cauchy-Schwarz inequality, we then obtain 
\begin{align*}
\left| \left\l \theta_{k^*} - \theta_p,  \theta_{k^*} - \hat \theta \right\r \right| 
&\le \gamma \sum_{k=p}^{k^*-1} \left\| \sum_{l'=1}^d v_{l' i_k}(Ax_k - y)_{l'} \right\| 
\left\|\sum_{l=1}^d v_{l i_k} (Ax_{k^*} - y)_l \right\| \displaybreak[0] \\
& \le \gamma \sum_{k=p}^{k^*-1} \sum_{l'=1}^d |v_{l' i_k}| \|(A x_k -y)_{l'}\|
\sum_{l=1}^d |v_{l i_k}| \|(Ax_{k^*} - y)_l\|  \displaybreak[0]\\
& \le \gamma \sum_{k=p}^{k^*-1} \|v_{i_k}\|^2 \|A x_k - y\| \|Ax_{k^*} - y\|  \displaybreak[0]\\
& \le \gamma v^* \sum_{k=p}^{k^*-1} \|A x_k - y\| \|Ax_{k^*} - y\|. 
\end{align*}
With the help of (\ref{RBCD.31}) and (\ref{RBCD.32}) we further obtain 
\begin{align*}
\left| \left\l \theta_{k^*} - \theta_p,  \theta_{k^*} - \hat \theta \right\r \right| 
& \le \gamma v^* \sum_{k=p}^{k^*-1} \|A x_k - y\|^2 \to 0
\end{align*}
as $p\rightarrow \infty$. Thus $\|\theta_p - \theta_{k^*} \|\rightarrow 0$ as $p\to\infty$. Similarly, 
we can also obtain $\| \theta_q - \theta_{k^*}\|\rightarrow 0$ as $p \rightarrow \infty$. Therefore, 
$\{\theta_k\}$ is a Cauchy sequence. 

Recall that $V$ is assumed to be of full column rank. Thus we can find a $b\times d$ matrix 
$U=(u_{i l})$ such that $UV = I_b$, where $I_b$ is the $b\times b$ identity matrix. Then
\begin{align*}
\sum_{l=1}^d u_{i l} \theta_{k,l} 
&= \sum_{l=1}^d u_{i l} \sum_{i'=1}^b v_{l i'} x_{k,i'} 
= \sum_{i'=1}^b \left(\sum_{l=1}^d u_{i l} v_{l i'} \right) x_{k,i'} 
= \sum_{i'=1}^b \delta_{i i'} x_{k,i'} = x_{k, i}
\end{align*}
for $i = 1, \cdots, b$. Hence we can recover $x_k$ from $\xi_k$. Let $\|U\|_F$ denote the 
Frobenius norm of $U$. Then by the Cauchy-Schwarz inequality we can obtain 
\begin{align*}
\| x_p - x_q \|^2 
& = \sum_{i=1}^b \left\|\sum_{l=1}^d u_{i l} (\theta_{p,l} - \theta_{q,l})\right\|^2
\le \sum_{i=1}^b \left(\sum_{l=1}^d |u_{i l}| \|\theta_{p,l} - \theta_{q,l}\|\right)^2 \\
&\le \left(\sum_{i=1}^b \sum_{l=1}^d | u_{i l} |^2\right) 
\left( \sum_{l=1}^d \| \theta_{p,l} - \theta_{q,l}\|^2 \right) \\
&\leq \|U\|_F^2 \| \theta_p - \theta_q \|^2
\end{align*}
which implies that $\{x_k \}$ is also a Cauchy sequence and hence $x_k \to x^*$ as $k\to \infty$ 
for some $x^* \in X^b$. Since $\|A x_k -y \| \to 0$ as $k\to \infty$, we can conclude $\|A x^* - y\| =0$, 
i.e. $x^*$ is a solution of $Ax = y$. 

The above argument actually shows that there is a random solution $x^*$ of $A x = y$ such that 
$x_k \to x^*$ as $k\to \infty$ along any sample path in $\Omega_2$. Since $\P(\Omega_2) = 1$, we 
have $x_k \to x^*$ as $k \to \infty$ almost surely. Note that Lemma 
\ref{RBCD4special.lem1} implies $\|\theta_k - \hat \theta\|\le \|\theta_0 - \hat \theta\|$, we 
can conclude that 
\begin{align*}
\|x_k - \hat x\|^2 \le \|U\|_F^2 \|\theta_k - \hat \theta\|^2 
\le \|U\|_F^2 \|\theta_0 - \hat \theta\|^2 
\end{align*}
which implies that $\{x_k \}$ is uniformly bounded in the sense of (\ref{bound}). Thus we may use 
Remark \ref{RBCD.rem1} to conclude that $x^* = x^\dag$ almost surely and hence $x_k \to x^\dag$ as 
$k\to \infty$ almost surely. Furthermore, by the dominated convergence theorem we can further obtain 
$\EE[\|x_k - x^\dag\|^2] \to 0$ as $k \to \infty$. The proof is therefore complete.
\end{proof}

\begin{remark}
{\rm 
In their exploration of the cyclic block coordinate descent method detailed in \cite{RNH2019} 
for solving the ill-posed problem (\ref{prob0}), the authors established the convergence of 
the generated sequence $\{x_k\}$ to a point $x^*\in \X$ satisfying
$$
\sum_{l=1}^d v_{li}(A x^* - y)_l =0, \quad i = 1, \cdots, b
$$
They then concluded that the full column rank property of $V$ implies $A x^* = y$. 
However, this condition on $V$ is not sufficient for drawing such a conclusion when 
$d >b$. We circumvented this issue for our RBCD method by leveraging Lemma \ref{RBCD.lem2}. 
Furthermore, when classifying the limit $x^*$, the authors in \cite{RNH2019} inferred that 
$x^*$ is the unique $x_0$-minimum-norm solution of (\ref{prob0}) by asserting that 
$x_{k+1} - x_k \in \mbox{Ran}(A^*)$ for all $k$. Regrettably, this assertion is inaccurate.
Actually 
$$
x_{k+1} - x_k \in \mbox{Ran}(A_1^*) \otimes \cdots \otimes \mbox{Ran}(A_b^*)
$$
which is considerably larger than $\mbox{Ran}(A^*)$. We demonstrated that the limit of our 
method is the $x_0$-minimum-norm solution by using Remark \ref{RBCD.rem1}. 
}
\end{remark}

\begin{theorem}
Consider Algorithm \ref{alg:RBCD} for solving (\ref{prob0}) with noisy data. Assume  
$V=(v_{l i})$ is of full column rank and define $v^*$ as in Lemma \ref{RBCD4special.lem1}. 
Assume $0 < \gamma < 2/(v^*\|K\|^2)$ and let $x^\dag$ denote the unique $x_0$-minimum-norm 
solution of (\ref{prob0}). If the integer $k_\d$ is chosen such that $k_\d \to \infty$ and 
$\d^2 k_\d\to 0$ as $\d \to 0$, then $\EE[\|x_{k_\d}^\d - x^\dag\|^2] \to 0$ as $\d \to 0$. 
\end{theorem}

\begin{proof}
By virtue of the inequality $\|a+b\|^2 \le 2 (\|a\|^2 + \|b\|^2)$ and Lemma \ref{RBCD.lem1}, we have 
\begin{align*}
\EE[\|x_{k_\d}^\d - x^\dag\|^2] 
& \le 2 \EE[\|x_{k_\d}^\d - x_{k_\d}\|^2] + 2 \EE[\|x_{k_\d} - x^\dag\|^2] \\
& \le \frac{2\gamma}{b} \d^2 k_\d + 2\EE[\|x_{k_\d} - x^\dag\|^2].  
\end{align*}
Thus, we may use the choice of $k_\d$ and Theorem \ref{RBCD.thm11} to conclude the proof. 
\end{proof}

The above theorem provides a convergence result on Algorithm \ref{alg:RBCD} under an 
{\it a priori} stopping rule when it is used to solve (\ref{prob0}). We can also apply 
Algorithm \ref{alg:RBCD_DP} to solve (\ref{prob0}). Correspondingly we have the 
following convergence result. 

\begin{theorem}
Consider Algorithm \ref{alg:RBCD_DP} for solving (\ref{prob0}) with noisy data. Assume  
$V=(v_{l i})$ is of full column rank and define $v^*$ as in Lemma \ref{RBCD4special.lem1}. 
Assume 
$$
0 < \gamma < \min\{2/(v^*\|K\|^2), (2-2/\tau)/\|A\|^2\}
$$ 
and let $x^\dag$ denote the unique $x_0$-minimum-norm solution of (\ref{prob0}). If the 
integer $k_\d$ is chosen such that $k_\d \to \infty$ as $\d \to 0$, then 
$\EE[\|x_{k_\d}^\d - x^\dag\|^2] \to 0$ as $\d \to 0$. 
\end{theorem}

\begin{proof}
Let $k \ge 0$ be any integer. Since $k_\d \to \infty$, we have $k_\d >k$ for small $\d>0$. 
According to (\ref{RBCD.DP1}) in Propositon \ref{RBCD.DP:prop1} we have 
$$
\EE[\|x_{k_\d}^\d - x^\dag\|^2] 
\le \EE[\|x_k^\d - x^\dag\|^2] \le 2 \EE[\|x_k^\d - x_k\|^2] + 2 \EE[\|x_k - x^\dag\|^2]. 
$$
By virtue of (\ref{RBCD.-1}) in Lemma \ref{RBCD.lem1}, we then obtain 
$$
\EE[\|x_{k_\d}^\d - x^\dag\|^2] 
\le \frac{4\gamma}{b} k \d^2 + 2 \EE[\|x_k - x^\dag\|^2]. 
$$
Therefore 
$$
\limsup_{\d\to 0} \EE[\|x_{k_\d}^\d - x^\dag\|^2] \le 2 \EE[\|x_k - x^\dag\|^2]. 
$$
Letting $k \to \infty$ and using Theorem \ref{RBCD.thm11}, we can conclude 
$\limsup_{\d\to 0} \EE[\|x_{k_\d}^\d - x^\dag\|^2] \le 0$ and hence 
$\EE[\|x_{k_\d}^\d - x^\dag\|^2] \to 0$ as $\d \to 0$. 
\end{proof}

\section{\bf Extension}\label{sect3}

In Algorithm \ref{alg:RBCD} we proposed a randomized block coordinate descent method for solving 
linear ill-posed problem (\ref{eq:problem}) and provided convergence analysis on the method which 
demonstrates that the iterates in general converge to the $x_0$-minimal norm solution. In many 
applications, however, we are interested in reconstructing solutions with other features, such as 
non-negativity, sparsity, and piece-wise constancy. Incorporating such feature information into 
algorithm design can significantly improve the reconstruction accuracy. Assume the feature of the 
sought solution can be detected by a convex function $\R: \X \to (-\infty, \infty]$. We may consider 
determining a solution $x^\dag$ of (\ref{eq:problem}) such that 
\begin{align}\label{Rmin}
\R(x^\dag) = \min\left\{\R(x): \sum_{i=1}^b A_i x_i = y\right\}.
\end{align}
We assume the following condition on $\R$.

\begin{assumption}\label{ass2}
$\R: \X \to (-\infty, \infty]$ is proper, lower semi-continuous and strongly convex in the sense 
that there is a constant $\kappa > 0$ such that 
\begin{align*}
\R(t x + (1-t) \bar x) + \kappa t (1-t) \|x - \bar x\|^2 \le t \R(x) + (1-t) \R(\bar x)
\end{align*}
for all $x, \bar x \in \emph{dom}(\R)$ and $0\le t \le 1$. Moreover, $\R$ has the separable structure  
$$
\R(x) := \sum_{i=1}^b \R_i(x_i), \qquad \forall x = (x_1, \cdots, x_b) \in \X,
$$
where each $\R_i$ is a function from $X_i$ to $(-\infty, \infty]$. 
\end{assumption}

Under Assumption \ref{ass2}, it is easy to see that each $\R_i$ is proper, lower semi-continuous, and 
strong convex from $X_i$ to $(-\infty, \infty]$. Furthermore, let $\p \R$ denote the subdifferential of 
$\R$, then the following facts on $\R$ hold (see \cite{JW2013,Z2002}): 

\begin{enumerate}
%\item[(i)] $\p \R(x) \ne \emptyset$ for all $x \in \X$, where $\p \R$ denotes the subdifferential of $\R$. 

%\item[(ii)] For any bounded set $K$ in $\X$ there is a constant $C_K$ such that $\|\xi\| \le C_K$ 
%for all $\xi \in \p \R(x)$ and $x \in K$. 

\item[(i)] For any $x \in \X$ with $\p \R(x) \ne \emptyset$ and $\xi \in \p \R(x)$ there holds 
\begin{align}\label{RBCD.62}
D_\R^\xi(\bar x, x) \ge \kappa \|\bar x - x\|^2, \quad \forall \bar x \in \X,
\end{align}
where  
\begin{align*}
D_\R^{\xi}(\bar x, x) := \R(\bar x) - \R(x) - \l \xi, \bar x - x\r, \quad \bar x \in \X
\end{align*}
is the Bregman distance induced by $\R$ at $x$ in the direction $\xi$. 

\item[(ii)] For any $x, \bar x \in \X$, $\xi\in \p \R(x)$ and $\bar \xi \in \p \R(\bar x)$ there holds 
$$
\l \xi - \bar \xi, x - \bar x\r \ge 2 \kappa \|x - \bar x\|^2. 
$$

\item[(iii)] For all $x, \bar x \in \X$, $\xi \in \p \R(x)$ and $\bar \xi \in \p \R(\bar x)$ there holds 
\begin{align}\label{RBCD.61}
D_\R^\xi (\bar x, x) \le \frac{1}{4\kappa} \|\bar \xi - \xi\|^2
\end{align}
\end{enumerate}

Let us elucidate how to extend the method (\ref{eq:RBCD}) to solve (\ref{eq:problem}) so that the 
convex function $\R$ can be incorporated to detect the solution feature. By introducing 
$g_k^\d := (g_{k,1}^\d, \cdots, g_{k,b}^\d) \in \X$ with 
$$
g_{k, i}^\d = \left\{\begin{array}{lll}
A_{i_k}^* (A x_k^\d - y^\d), & \mbox{ if } i = i_k,\\
0,  & \mbox{ otherwise}
\end{array}\right.
$$
we can rewrite (\ref{eq:RBCD}) as 
\begin{align*}
x_{k+1}^\d & = x_k^\d - \gamma g_k^\d 
= \arg\min_{x\in \X} \left\{\frac{1}{2} \|x- (x_k^\d - \gamma g_k^\d)\|^2\right\} \\
& = \arg\min_{x\in \X} \left\{\frac{1}{2} \|x-x_k^\d\|^2 + \gamma\l g_k^\d, x\r\right\}. 
\end{align*}
Assume $\p \R(x_k^\d) \ne \emptyset$ and take $\xi_k^\d \in \p \R(x_k^\d)$, we may use the Bregman
distance $D_\R^{\xi_k^\d}(x, x_k^\d)$ to replace $\frac{1}{2}\|x-x_k^\d\|^2$ in the above equation 
to obtain the new updating formula
\begin{align*}
x_{k+1}^\d & = \arg\min_{x\in \X} \left\{D_\R^{\xi_k^\d}(x, x_k^\d) + \gamma \l g_k^\d, x\r\right\}
= \arg\min_{x\in \X} \left\{\R(x) - \l \xi_k^\d - \gamma g_k^\d, x\r\right\}. 
\end{align*}
Letting $\xi_{k+1}^\d := \xi_k^\d - \gamma g_k^\d$, then we have 
$$
x_{k+1}^\d = \arg\min_{x\in \X} \left\{\R(x) - \l \xi_{k+1}^\d, x\r \right\}. 
$$
Recall the definition of $g_k^\d$, we can see that $\xi_{k+1}^\d = (\xi_{k+1,1}^\d, \cdots, \xi_{k+1,b}^\d)$ with
\begin{align}\label{RBCD.72}
\xi_{k+1, i}^\d = \left\{\begin{array}{lll}
\xi_{k, i_k}^\d - \gamma A_{i_k}^*(A x_k^\d - y^\d), & \mbox{ if } i = i_k,\\
\xi_{k, i}^\d & \mbox{ otherwise}.
\end{array}\right. 
\end{align}
Under Assumption \ref{ass2}, it is known that $x_{k+1}^\d$ is uniquely defined and 
$\xi_{k+1}^\d \in \p \R(x_{k+1}^\d)$. According to the separable structure of $\R$, 
it is easy to see that $x_{k+1}^\d = (x_{k+1,1}^\d, \cdots, x_{k+1, b}^\d)$ with 
\begin{align}\label{RBCD.73}
x_{k+1, i}^\d = \left\{\begin{array}{lll}
\displaystyle{\arg\min_{z\in X_{i_k}} \left\{\R_{i_k}(z) - \l \xi_{k+1, i_k}^\d, z\r\right\}}, & \mbox{ if } i = i_k,\\
x_{k, i}^\d, & \mbox{ otherwise}.
\end{array}\right. 
\end{align}
Since $\xi_{k+1}^\d \in \p \R(x_{k+1}^\d)$, we may repeat the above procedure. This leads us to 
propose the following algorithm. 

\begin{algorithm}[RBCD method with convex regularization function]\label{alg:RBCD+}
Let $\xi_0 = 0 \in \X$ and define $x_0 := \arg\min_{x\in \X} \R(x)$ as an initial guess.
Set $\xi_0^\d := \xi_0$, $x_0^\d := x_0$ and calculate $r_0^\d:= A x_0^\d - y^\d$. 
Choose a suitable step size $\gamma>0$. For all integers $k \ge 0$ 
do the following:  
\begin{enumerate}
\item[(i)] Pick an index $i_k \in \{1, \cdots, b\}$ randomly via the uniform distribution; 

\item[(ii)] Update $\xi_{k+1}^\d$ by the equation (\ref{RBCD.72}), i.e. setting $\xi_{k+1,i}^\d = \xi_{k,i}^\d$ 
for $i \ne i_k$ and 
$$
\xi_{k+1, i_k}^\d = \xi_{k, i_k}^\d - \gamma A_{i_k}^* r_k^\d;
$$

\item[(iii)] Update $x_{k+1}^\d$ by the equation (\ref{RBCD.73}); 

\item[(iv)] Calculate $r_{k+1}^\d = r_k^\d + A_{i_k}(x_{k+1, i_k}^\d - x_{k,i_k}^\d)$. 
\end{enumerate}
\end{algorithm}

The analysis on Algorithm \ref{alg:RBCD+} is rather challenging. In the following we will prove
some results which support the use of Algorithm \ref{alg:RBCD+} to solve ill-posed problems. We 
start with the following lemma. 

\begin{lemma}\label{RBCD:lem3.2}
Let Assumption \ref{ass2} hold and consider Algorithm \ref{alg:RBCD+}. Then for any solution $\hat x$ 
of (\ref{eq:problem}) in $\emph{dom}(\R)$ there holds 
\begin{align*}
\EE\left[D_\R^{\xi_{k+1}^\d}(\hat x, x_{k+1}^\d)\Big|\F_k\right] - D_\R^{\xi_k^\d}(\hat x, x_k^\d) 
& \le - c_3 \frac{\gamma}{b} \|A x_k^\d - y^\d\|^2
+ \frac{\gamma}{b} \d \|A x_k^\d - y^\d\|
\end{align*}
for all integers $k \ge 0$, where $c_3 := 1- \gamma\|A\|^2/(4 \kappa)$.
\end{lemma}

\begin{proof}
For any solution $\hat x$ of (\ref{eq:problem}) in $\mbox{dom}(\R)$ let 
$\Delta_k^\d := D_\R^{\xi_k^\d}(\hat x, x_k^\d)$ for all integers $k$. Then, by using (\ref{RBCD.61}) 
and the definition of $\xi_{k+1}^\d$, we have 
\begin{align*}
\Delta_{k+1}^\d - \Delta_k^\d 
& = D_\R^{\xi_{k+1}^\d}(x_k^\d, x_{k+1}^\d) + \l \xi_{k+1}^\d - \xi_k^\d, x_k^\d - \hat x\r \\
& \le \frac{1}{4 \kappa} \|\xi_{k+1}^\d - \xi_k^\d\|^2 + \l \xi_{k+1}^\d - \xi_k^\d, x_k^\d - \hat x\r \\
& = \frac{1}{4 \kappa} \|\xi_{k+1, i_k}^\d - \xi_{k,i_k}^\d\|^2 
+ \l \xi_{k+1, i_k}^\d - \xi_{k, i_k}^\d, x_{k, i_k}^\d - \hat x_{i_k}\r \\
& = \frac{\gamma^2}{4 \kappa} \|A_{i_k}^*(A x_k^\d - y^\d)\|^2 
- \gamma \l A x_k^\d - y^\d, A_{i_k}(x_{k, i_k}^\d - \hat x_{i_k}^\d)\r.
\end{align*}
By taking the conditional expectation on $\F_k$ and using $\|y^\d - y\| \le \d$ we can obtain  
\begin{align*}
\EE[\Delta_{k+1}^\d|\F_k] - \Delta_k^\d 
& \le \frac{\gamma^2}{4 \kappa b} \sum_{i=1}^b \|A_i^*(A x_k^\d - y^\d)\|^2 
- \frac{\gamma}{b} \l A x_k^\d - y^\d, A x_k^\d - y\r \nonumber \\
& = \frac{\gamma^2}{4 \kappa b} \|A^*(A x_k^\d - y^\d)\|^2 - \frac{\gamma}{b} \|A x_k^\d - y^\d\|^2 
- \frac{\gamma}{b} \l A x_k^\d - y^\d, y^\d - y\r \nonumber \\
& \le - \left(1- \frac{\gamma\|A\|^2}{4 \kappa}\right) \frac{\gamma}{b} \|A x_k^\d - y^\d\|^2 
+ \frac{\gamma}{b} \d \|A x_k^\d - y^\d\|.
\end{align*}
The proof is therefore complete. 
\end{proof}

\begin{theorem}\label{RBCD.thm21}
Let $\X$ be finite dimensional and let $\R: \X \to {\mathbb R}$ satisfy Assumption \ref{ass2}. 
Consider Algorithm \ref{alg:RBCD+} with the exact data. If $0< \gamma <4 \kappa/\|A\|^2$, then 
$\{x_k\}$ converges to a random solution $\bar x$ of (\ref{eq:problem}) almost surely. 
\end{theorem}

\begin{proof}
Since $\X$ is finite dimensional and $\R$ maps $\X$ to ${\mathbb R}$, the convex function $\R$ 
is Lipschitz continuous on bounded sets and $\p \R(x) \ne \emptyset$ for all $x \in \X$; moreover, 
for any bounded set $K \subset \X$ there is a constant $C_K$ such that $\|\xi\| \le C_K$ for all 
$\xi \in \p \R(x)$ and $x \in K$. 

In the following we will use the similar argument in the proof of Theorem \ref{RBCD.thm1} to show 
the almost sure convergence of $\{x_k\}$. Let $S$ denote the set of solutions of (\ref{eq:problem})
in $\mbox{dom}(\R)$. By using Lemma \ref{RBCD:lem3.2} with exact data and Proposition \ref{prop:DMT} 
we can conclude for any $z \in S$ that the event 
$$
\tilde \Omega_z:= \left\{\lim_{k\to \infty} D_\R^{\xi_k}(z, x_k) \mbox{ exists  and is finite}\right\}
$$
has probability one. Since $\X$ is separable, we can find a countable set $D\subset S$ 
such that $D$ is dense in $S$. Let 
$$
\tilde \Omega_1 := \bigcap_{z\in D} \tilde \Omega_z.
$$
Then $\P(\tilde \Omega_1) = 1$. We now show that, for any $\tilde z\in S$, along any sample path 
in $\tilde \Omega_1$ the sequence $\{D_\R^{\xi_k}(\tilde z, x_k)\}$ is convergent. To see this, 
we take a sequence $\{z_l\}\subset D$ such that $z_l \to \tilde z$ as $l \to \infty$. For any 
sample path $\omega \in \tilde \Omega_1$ we have  
$$
D_\R^{\xi_k(\omega)} (\tilde z, x_k(\omega)) - D_\R^{\xi_k(\omega)} (z_l, x_k(\omega)) 
= \R(\tilde z) - \R(z_l) - \l \xi_k(\omega), \tilde z - z_l\r. 
$$
Since $\{D_\R^{\xi_k(\omega)}(z_l, x_k(\omega))\}$, for a fixed $l$, is convergent, it is bounded. 
By (\ref{RBCD.62}), $\{x_k(\omega)\}$ is then bounded and hence we can find a constant $M$ 
such that $\|\xi_k(\omega)\| \le M$ for all $k$. Therefore 
\begin{align*}
\R(\tilde z) - \R(z_l) - M\|\tilde z - z_l\| 
& \le D_\R^{\xi_k(\omega)} (\tilde z, x_k(\omega)) - D_\R^{\xi_k(\omega)} (z_l, x_k(\omega)) \\
& \le  \R(\tilde z) - \R(z_l) + M\|\tilde z - z_l\|. 
\end{align*}
This together with the existence of $\lim_{k\to \infty} D_\R^{\xi_k(\omega)} (z_l, x_k(\omega))$ implies 
$$
0 \le \limsup_{k\to \infty} D_\R^{\xi_k(\omega)} (\tilde z, x_k(\omega))
- \liminf_{k\to \infty} D_\R^{\xi_k(\omega)} (\tilde z, x_k(\omega))
\le 2 M \|z_l - \tilde z\|.
$$
Letting $l \to \infty$ then shows that $\lim_{k\to \infty} D_\R^{\xi_k(\omega)} (\tilde z, x_k(\omega))$ 
exists and is finite for every $\omega \in \tilde \Omega_1$ and $\tilde z \in S$. 

Next we show the almost sure convergence of $\{x_k\}$. By using Lemma \ref{RBCD:lem3.2} with exact data 
we can conclude the event  
\begin{align*}
\tilde \Omega_2:= \left\{ \sum_{k=0}^\infty \|A x_k - y\|^2 < \infty\right\}
\end{align*}
has probability one. Let $\tilde \Omega_0:= \tilde \Omega_1 \cap \tilde \Omega_2$. Then 
$\P(\tilde \Omega_0) = 1$. Note that, along any sample path in $\tilde \Omega_0$, 
$\{D_\R^{\xi_k}(z, x_k)\}$ is convergent for any $z \in S$ 
and thus $\{x_k\}$ and $\{\xi_k\}$ are bounded. Thus, we can find a subsequence $\{k_l\}$ of 
positive integers such that $x_{k_l}\to \bar x$ and $\xi_{k_l} \to \bar \xi$ as $l \to \infty$. 
Because $\xi_{k_l} \in \p \R(x_{k_l})$, we have $\bar \xi \in \p \R(\bar x)$. Since 
$\|A x_k - y\| \to 0$ as $k \to \infty$, we can obtain $A \bar x = y$, i.e. $\bar x \in S$. 
Consequently $\{D_\R^{\xi_k} (\bar x, x_k)\}$ is convergent. To show $x_k \to \bar x$, it suffices 
to show that $\bar x$ is the unique cluster point of $\{x_k\}$. Let $x^*$ be any cluster point of 
$\{x_k\}$. Then, by using the boundedness of $\{\xi_k\}$, there is another subsequence $\{n_l\}$ 
of positive integers such that $x_{n_l} \to x^*$ and $\xi_{n_l} \to \xi^*$ as $l \to \infty$. Then 
$\xi^* \in \p \R(x^*)$, $x^* \in S$ and thus $\{D_\R^{\xi_k}(x^*, x_k)\}$ is convergent. Noting the 
identity
$$
\l \xi_k, x^* - \bar x\r = D_\R^{\xi_k}(\bar x, x_k) - D_\R^{\xi_k}(x^*, x_k) + \R(x^*) - \R(\bar x),  
$$
we can conclude that $\lim_{k\to \infty} \l \xi_k, x^*- \bar x\r$ exists. Therefore 
\begin{align*}
\lim_{k\to \infty} \l \xi_k, x^* - \bar x\r & = \lim_{l\to \infty} \l \xi_{k_l}, x^*- \bar x\r 
= \l \bar \xi, x^* - \bar x\r, \\
\lim_{k\to \infty} \l \xi_k, x^* - \bar x\r & = \lim_{l\to \infty} \l \xi_{n_l}, x^* - \bar x\r 
= \l \xi^*, x^* - \bar x\r
\end{align*}
and thus $\l \xi^* - \bar \xi, x^* - \bar x\r = 0$. Since $\bar \xi \in \p \R(\bar x)$ and 
$\xi^*\in \p \R(x^*)$, we may use the strong convexity of $\R$ to conclude $x^* = \bar x$. The proof 
is complete. 
\end{proof}

For Algorithm \ref{alg:RBCD+} with noisy data, we have the following convergence result under
an {\it a priori} stopping rule. 

\begin{theorem}\label{RBCD.thm22}
Let $\X$ be finite dimensional and let $\R: \X \to {\mathbb R}$ satisfy Assumption \ref{ass2}. 
Consider Algorithm \ref{alg:RBCD+} with $0< \gamma <4 \kappa/\|A\|^2$. Let $\{x_k\}$ be the random 
sequence determined by Algorithm \ref{alg:RBCD+} with the exact data. If $\{x_k\}$ is uniformly 
bounded in the sense of (\ref{bound}) and if the integer $k_\d$ is chosen such that $k_\d \to \infty$ 
and $\d^2 k_\d\to 0$ as $\d \to 0$, then
$$
\EE\left[\|x_{k_\d}^\d - x^\dag\|^2\right] \to 0 
$$
as $\d \to 0$, where $x^\dag$ denotes the unique solution of (\ref{eq:problem}) satisfying 
(\ref{Rmin}). 
\end{theorem}

\begin{proof}
According to Theorem \ref{RBCD.thm21}, there is a random solution $\bar x$ of (\ref{eq:problem})
such that $x_k \to \bar x$ as $k \to \infty$ almost surely. Since $\{x_k\}$ is assumed to satisfy
(\ref{bound}), we can use the dominated convergence theorem to conclude $\EE[\|x_k - \bar x\|^2] \to 0$
as $k \to \infty$. Based on this, we will show 
\begin{align}\label{RBCD+.0}
\lim_{k\to \infty} \EE[\|x_k - x^\dag\|^2] = 0.  
\end{align} 
Indeed, by using a similar argument in Remark \ref{RBCD.rem1} and noting $\xi_0 =0$ we can conclude that 
\begin{align}\label{RBCD+.1}
\EE[\l \xi_k, \bar x - x^\dag\r] =0, \quad \forall k. 
\end{align}
According to the definition of $x^\dag$, the convex function $\varphi(t):= \R(x^\dag + t (\bar x - x^\dag))$
on ${\mathbb R}$ achieves its minimum at $t=0$ and thus $0 \in \p \varphi(0)$. Note that 
$$
\p \varphi(0) = \{\l \xi, \bar x - x^\dag\r: \xi \in \p \R(x^\dag)\}.
$$
Consequently, there exists $\xi^\dag \in \p \R(x^\dag)$ such that $\l \xi^\dag, \bar x - x^\dag\r =0$. 
This $\xi^\dag$ depends on $\bar x$ and hence it is a random element in $\p \R(x^\dag)$. Nevertheless, 
$\EE[\l \xi^\dag, \bar x - x^\dag\r] = 0$. Combining this with (\ref{RBCD+.1}) gives 
$$
0 = \EE[\l \xi_k - \xi^\dag, \bar x - x^\dag\r] 
= \EE[\l \xi_k - \xi^\dag, x_k - x^\dag\r] + \EE[\l \xi_k - \xi^\dag, \bar x - x_k\r]
$$
for all integers $k \ge 0$. Since $\{x_k\}$ is uniformly bounded in the sense of (\ref{bound}) and $\R$ 
is Lipschitz continuous on bounded sets, we can find a constant $M$ such that $\|\xi^\dag\|\le M$ 
and $\|\xi_k\|\le M$ for all $k$ almost surely. Thus 
\begin{align*}
\left|\EE[\l \xi_k - \xi^\dag, \bar x - x_k\r]\right| 
\le 2 M \EE[\|\bar x - x_k\|] \le 2 M (\EE[\|x_k - \bar x\|^2])^{1/2} \to 0 
\end{align*}
as $k \to \infty$. Consequently 
$$
\lim_{k\to \infty} \EE[|\l \xi_k - \xi^\dag, x_k - x^\dag\r| = 0. 
$$
By the strong convexity of $\R$, we have $\l \xi_k - \xi^\dag, x_k - x^\dag\r 
\ge 2 \kappa \|x_k - x^\dag\|^2$. Therefore $\lim_{k\to \infty} \EE[\|x_k - x^\dag\|^2] =0$
which shows (\ref{RBCD+.0}). 

Next, by using the definition of $\xi_k^\d, x_k^\d$ and $\xi_k, x_k$, it is easy to show by an 
induction argument that, along any sample path there holds 
\begin{align}\label{RBCD+.2}
\xi_k^\d \to \xi_k \quad \mbox{and} \quad x_k^\d \to x_k \quad \mbox{as } \d \to 0
\end{align}
for every integer $k\ge 0$. 

Finally, we show $\EE[\|x_{k_\d}^\d - x^\dag\|^2] \to 0$ as $\d \to 0$ if $k_\d$ is chosen 
such that $k_\d\to \infty$ and $\d^2 k_\d \to 0$ as $\d \to 0$. To see this, let 
$\Delta_k^\d := D_\R^{\xi_k^\d}(x^\dag, x_k^\d)$ for all $k \ge 0$. From Lemma \ref{RBCD:lem3.2} 
it follows that 
\begin{align*}
\EE[\Delta_{k+1}^\d] 
\le \EE[\Delta_k^\d] + \frac{\gamma}{4b c_3} \d^2, \quad \forall k \ge 0.
\end{align*}
Let $k\ge 0$ be any fixed integer. Since $k_\d \to \infty$, we have $k_\d >k$ for small $\d>0$. 
Consequently, by virtue of the above inequality, we can obtain 
\begin{align*}
\EE[\Delta_{k_\d}^\d] \le  \EE[\Delta_k^\d] + \frac{\gamma}{4b c_3} (k_\d - k) \d^2 
\le \EE[\Delta_k^\d] + \frac{\gamma}{4 b c_3} k_\d \d^2
\end{align*}
Since $\d^2 k_\d \to 0$ as $\d \to 0$, we may use (\ref{RBCD+.2}) and the continuity of $\R$ to 
further obtain 
\begin{align*}
\limsup_{\d\to 0} \EE[\Delta_{k_\d}^\d] 
\le \limsup_{\d\to 0} \EE[\Delta_k^\d] = \EE\left[D_\R^{\xi_k}(x^\dag, x_k)\right].
\end{align*}
Since $\{x_k\}$ satisfies (\ref{bound}) and $\R$ is Lipschitz continuous on bounded sets, there 
is a constant $L$ such that 
\begin{align*}
D_\R^{\xi_k}(x^\dag, x_k) 
& = \R(x^\dag) - \R(x_k) - \l \xi_k, x^\dag - x_k\r \\
& \le |\R(x^\dag) - \R(x_k)| + \|\xi_k\| \|x^\dag - x_k\| \\
& \le (L + M) \|x^\dag - x_k\|
\end{align*}
Therefore 
\begin{align*}
\limsup_{\d\to 0} \EE[\Delta_{k_\d}^\d] 
\le (L+M) \EE[\| x_k - x^\dag\|] \le (L+M) (\EE[\|x_k - x^\dag\|^2])^{1/2}
\end{align*}
for all $k \ge 0$. By taking $k \to \infty$ and using (\ref{RBCD+.0}) we thus have 
$\limsup_{\d\to 0} \EE[\Delta_{k_\d}^\d] \le 0$ which implies $\EE[\Delta_{k_\d}^\d] \to 0$
as $\d \to 0$. Consequently, it follows from the strong convexity of $\R$ that 
$\EE[\|x_{k_\d}^\d - x^\dag\|^2] \to 0$ as $\d \to 0$. 
\end{proof}

In the above results we considered Algorithm \ref{alg:RBCD+} by terminating the iterations
via {\it a priori} stopping rules. Analogous to Algorithm \ref{alg:RBCD_DP}, we may consider 
incorporating the discrepancy principle into Algorithm \ref{alg:RBCD+}. This leads to the 
following algorithm. 

\begin{algorithm}\label{alg:RBCD+DP}
Let $\xi_0 = 0 \in \X$ and define $x_0 := \arg\min_{x\in \X} \R(x)$ as an initial guess.
Set $\xi_0^\d := \xi_0$, $x_0^\d := x_0$ and calculate $r_0^\d:= A x_0^\d - y^\d$. 
Choose $\tau>1$ and $\gamma>0$. For all integers $k \ge 0$ do the following:  
\begin{enumerate}
\item[(i)] Set the step-size $\gamma_k$ by 
\begin{align*}
\gamma_k := \left\{\begin{array}{lll}
\gamma & \mbox{ if } \|r_k^\d\| > \tau \d,\\
0 & \mbox{ if } \|r_k^\d\| \le \tau \d;
\end{array}\right.
\end{align*}

\item[(ii)] Pick an index $i_k \in \{1, \cdots, b\}$ randomly via the uniform distribution; 

\item[(iii)] Update $\xi_{k+1}^\d$ by the equation (\ref{RBCD.72}), i.e. setting $\xi_{k+1,i}^\d = \xi_{k,i}^\d$ 
for $i \ne i_k$ and 
$$
\xi_{k+1, i_k}^\d = \xi_{k, i_k}^\d - \gamma_k A_{i_k}^* r_k^\d;
$$

\item[(iv)] Update $x_{k+1}^\d$ by the equation (\ref{RBCD.73}); 

\item[(v)] Calculate $r_{k+1}^\d = r_k^\d + A_{i_k}(x_{k+1, i_k}^\d - x_{k,i_k}^\d)$. 
\end{enumerate}
\end{algorithm}

For Algorithm \ref{alg:RBCD+DP} we have the following result which in particular shows 
that the iteration must terminate in finite many steps almost surely, i.e. there is an 
event of probability one such that, along any sample path in this event,  $\gamma_k = 0$ 
for sufficiently large $k$. 

\begin{theorem}
Let Assumption \ref{ass2} hold and consider Algorithm \ref{alg:RBCD+DP} with
$\gamma = \mu/\|A\|^2$ for some $0 < \mu < 4 \kappa (1-1/\tau)$. Then the iteration  
must terminate in finite many steps almost surely. If, in addition, all the conditions 
in Theorem \ref{RBCD.thm22} hold, then for any integer $k_\d$ with $k_\d\to \infty$ 
as $\d \to 0$ there holds 
\begin{align}\label{RBCD+.4}
\lim_{\d\to 0} \EE[\|x_{k_\d}^\d - x^\dag\|^2] =0,
\end{align}
where $x^\dag$ denotes the unique solution of (\ref{eq:problem}) satisfying (\ref{Rmin}). 
\end{theorem}

\begin{proof}
For any integer $k \ge 0$ let $\Delta_k^\d:= D_\R^{\xi_k^\d} (x^\dag, x_k^\d)$. 
By noting that $\gamma_k$ is $\F_k$-measurable, we may use the similar argument in the proof 
of Lemma \ref{RBCD:lem3.2} to obtain 
\begin{align*}
\EE[\Delta_{k+1}^\d|\F_k] -\Delta_k^\d 
& \le - \frac{1}{b}\left(1 - \frac{\gamma_k \|A\|^2}{4\kappa}\right) \gamma_k \|A x_k^\d - y^\d\|^2 
+ \frac{\gamma_k}{b} \d \|A x_k^\d - y^\d\|. 
\end{align*}
According to the definition of $\gamma_k$, we have $\gamma_k \le \gamma$ and 
$\gamma_k \d \le \frac{\gamma_k}{\tau} \|A x_k^\d - y^\d\|$. Therefore 
\begin{align*}
\EE[\Delta_{k+1}^\d|\F_k] -\Delta_k^\d 
& \le - \frac{1}{b}\left(1 - \frac{1}{\tau} - \frac{\gamma_k \|A\|^2}{4\kappa}\right) 
\gamma_k \|A x_k^\d - y^\d\|^2 \\
& \le - \frac{1}{b}\left(1 - \frac{1}{\tau} - \frac{\mu}{4\kappa}\right) 
\gamma_k \|A x_k^\d - y^\d\|^2
\end{align*}
Taking the full expectation then gives
\begin{align}\label{RBCD+.DP1}
\EE[\Delta_{k+1}^\d] \le \EE[\Delta_k^\d] - c_4 \EE\left[\gamma_k \|A x_k^\d - y^\d\|^2\right] 
\end{align}
for all $k\ge 0$, where $c_4: = (1-1/\tau - \mu/(4\kappa))/b > 0$.

Based on (\ref{RBCD+.DP1}) we now show that the method must terminate after finite many 
steps almost surely. To see this, consider the event
$$
{\mathcal E} := \left\{\|A x_k^\d - y^\d \|>\tau \d \mbox{ for all integers } k \ge 0\right\}
$$
It suffices to show $\P({\mathcal E}) = 0$. By virtue of (\ref{RBCD+.DP1}) we have 
\begin{align*}
c_4 \EE\left[\gamma_k \|A x_k^\d - y^\d\|^2\right] 
\le \EE[\Delta_k^\d] - \EE[\Delta_{k+1}^\d] 
\end{align*}
and hence for any integer $n \ge 0$ that 
\begin{align}\label{RBCD.DP2}
c_4 \sum_{k=0}^n \EE\left[\gamma_k \|A x_k^\d - y^\d\|^2\right]
\le \EE[\Delta_0^\d] = \R(\bar x) - \R(x_0) < \infty. 
\end{align}
Let $\chi_{\mathcal E}$ denote the characteristic function of ${\mathcal E}$, i.e.
$\chi_{\mathcal E}(\omega) =1$ if $\omega \in {\mathcal E}$ and $0$ otherwise. Then 
\begin{align*}
\EE\left[\gamma_k \|A x_k^\d - y^\d\|^2\right]
\ge \EE\left[\gamma_k \|A x_k^\d - y^\d\|^2\chi_{\mathcal E}\right] 
\ge \gamma \tau^2 \d^2 \EE[\chi_{\mathcal E}]
= \gamma \tau^2 \d^2 \P({\mathcal E}).
\end{align*}
Combining this with (\ref{RBCD.DP2}) gives 
$$
c_4 \gamma \tau^2 \d^2 (n+1) \P({\mathcal E}) \le \R(\bar x) - \R(x_0)
$$
for all $n \ge 0$ and hence $\P({\mathcal E}) \le (\R(\bar x) - \R(x_0))/(c_4 \gamma \tau^2 \d^2 (n+1))\to 0$ 
as $n \to \infty$. Thus $\P({\mathcal E}) =0$. 

Next we show (\ref{RBCD+.4}) under the conditions given in Theorem \ref{RBCD.thm22}. Let 
$k \ge 0$ be any integer. Since $k_\d \to \infty$, we have $k_\d >k$ for small $\d>0$. 
By virtue of (\ref{RBCD+.DP1}) and analogous to the proof of Theorem \ref{RBCD.thm22} we 
can obtain 
\begin{align*}
\limsup_{\d\to 0} \EE[\Delta_{k_\d}^\d] \le \EE\left[D_\R^{\xi_k}(x^\dag, x_k)\right]
\le C \left(\EE[\|x_k - x^\dag\|^2]\right)^{1/2},
\end{align*}
where $C$ is a generic constant independent of $k$. Letting $k \to \infty$ and using (\ref{RBCD+.0}) we 
thus obtain $\EE[\Delta_{k_\d}^\d] \to 0$ as $\d \to 0$ which together with the strong convexity of $\R$ 
implies (\ref{RBCD+.4}). The proof is therefore complete. 
\end{proof}

\section{\bf Numerical results}\label{sect4}

In this section, we present numerical simulations to verify the theoretical results 
and test the performance of the RBCD method for solving some specific linear ill-posed 
problems that arise from two imaging modalities, including X-ray computed tomography (CT) 
and coded aperture compressive temporal imaging. In all simulations, the noisy data 
$y^{\delta}$ is generated from the exact data $y$ by
$$
y^{\delta} = y + \delta_{\text{rel}} \|y\| \xi,
$$
where $\delta_{\text{rel}}$ is the relative noise level and $\xi$ is a normalized random 
noise obeying the standard Gaussian distribution (i.e., $\|\xi\| = 1$), then the noisy 
level $\delta = \delta_{\text{rel}} \|y\|$. The experiments  are carried out by \texttt{MATLAB} 
2021a on a laptop computer (1.80Hz Intel Core i7 processor with 16GB random access memory). 

\subsection{Computed tomography}
As the first testing example, we consider the standard 2D parallel beam X-ray CT from 
stimulated tomographic data to illustrate the performance of the RBCD method in this 
classical imaging task. The process in this modality consists of reconstructing a slice 
through the patient's body by collecting the attenuation of X-ray dose as they pass 
through the biological issues, which can be mathematically expressed as finding a compactly 
supported function from its line integrals, the so-called Radon transform.

In this example, we discretize the sought image into a 256×256 pixel grid and use the 
function \texttt{paralleltomo} from the package AIR TOOLS \cite{HS2012} in \texttt{MATLAB} 
to generate the discrete model. The true images we used in our experiments are the 
Shepp-Logan phantom and real chest CT image (see Figure \ref{fig:CTimage}). The real 
chest CT image is provided by the dataset \texttt{chestVolume} in \texttt{MATLAB}. If 
we consider the reconstruction from the tomographic data with $p$ projections and 367 
X-ray lines per projection, it leads to a linear system $Ax=y$ where $A$ is a coefficient 
matrix with the size $(367\times p)\times 65536$. Let $x^\dagger$ be the true image vector 
by stacking all the columns of the original image matrix and $y^\dagger= Ax^\dagger$ be 
the exact data. We divide the true solution $x^\dagger$ into $b$ blocks equally, and then 
the coefficient matrix $A$ also can be divided into $b$ blocks by column correspondingly.  

\begin{figure}[htb!]
    \centering
    \subfigure{
    \includegraphics[scale=0.45]{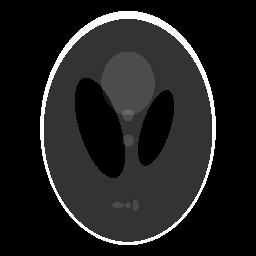}
}
\subfigure{
    \includegraphics[scale=0.45]{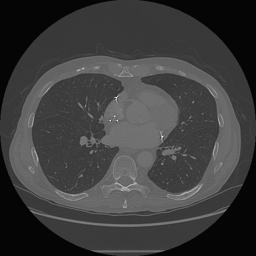}
} 
    \caption{The Sheep-Logan phantom (left) and a real CT image (right).}  
    \label{fig:CTimage} 
\end{figure}

Assuming the image to be reconstructed is the Shepp-Logan phantom, we compare the RBCD 
method with the cyclic block coordinate descent (CBCD) method studied in \cite{RNH2019} 
for solving the full angle CT problem from the exact tomographic data with 90 projection 
angles evenly distributed between 1 and 180 degrees. In theory, the inversion problem with 
the complete tomographic data is mildly ill-posed. To make a fair comparison, we run both 
the RBCD method (Algorithm \ref{alg:RBCD}) and the CBCD method 4000 iteration steps with 
the common parameters $b=8$, $\gamma = \mu/\|A\|^2$ with $\mu = 0.5$, and the intial guess 
$x_0 = 0$. The reconstruction results are reported in the left and middle figures of Figure 
\ref{fig:shepp-logan}, which shows two methods produce similar results by the noiseless 
tomographic data. To clarify the difference of two methods, we also plot the relative square 
errors $\|x_n-x^\dagger\|_{L^2}^2/\|x^\dagger\|_{L^2}^2$ of CBCD method with 500 iterations 
and the relative mean square errors $\EE[\|x_{n}-x^\dagger\|_{L^2}^2/\|x^\dagger\|_{L^2}^2]$ 
which are calculated approximately by the average of 100 independent runs of the RBCD method 
with 500 iterations in the right figure of Figure \ref{fig:shepp-logan}. To the best of our 
knowledge, for the general linear ill-posed problems (\ref{eq:problem}) by using the CBCD 
method, the convergence result by noiseless data and regularization property by noisy data 
have not been established so far.  

\begin{figure}[htb!]
    \centering
\subfigure{
    \includegraphics[scale=0.45]{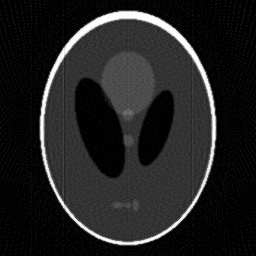}
}
\subfigure{
    \includegraphics[scale=0.45]{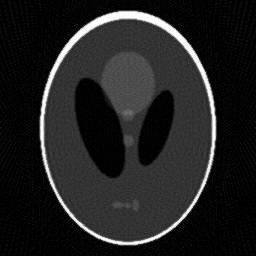}
}
\subfigure{
    \includegraphics[scale=0.2]{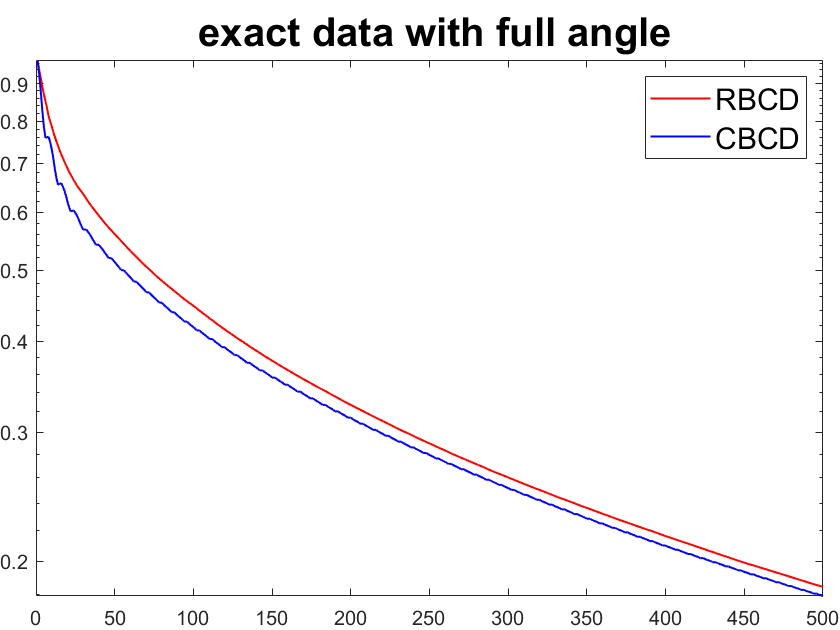}
}
    \caption{The reconstruction results by the RBCD method (left) and CBCD method (middle), 
    and the numerical reconstruction errors of two methods (right) from the exact full 
    angle tomographic data of the Shepp-Logan phantom, respectively.}    
    \label{fig:shepp-logan}
\end{figure}

\begin{table}[htb]
    \centering
    \begin{tabular}{lllllll}
    \toprule
        block number &  1  & 2 & 4 & 8 & 16\\
    \midrule   
        iteration number & 202 & 205 & 424 & 870 & 1819 \\
        Time (s) & 9.4262 & 4.9443 & 5.1250 & 5.5956 & 6.5644\\      
    \bottomrule
    \end{tabular}
    \caption{The reconstruction results by the Algorithm \ref{alg:RBCD} with different
    block numbers from the exact full angle tomographic data of the Shepp-Logan phantom.}
    \label{tab:Landweber_vs_RBCD}
\end{table}

To further compare the RBCD method with the Landweber iteration, we consider Algorithm 
\ref{alg:RBCD} for various block numbers $b=1,2,4,8,16$ to solve the above same problem 
using the same exact tomographic data. Note that Landweber iteration can be seen as the 
special case of Algorithm \ref{alg:RBCD} with $b=1$. To give a fair comparison, for all 
experiments with different block numbers we use the same step-size choose $\gamma = \mu/\|A\|^2$ 
with $\mu = 1.99$ and terminate the algorithm when the relative errors satisfy 
$\|x_k^\d - x^\dag\|_{L^2}^2/\| x^\dag\|_{L^2}^2 < 0.05$ for the first time. We perform 
100 independent runs and calculate the average of the required iteration number and 
computational time. The results of these experiments are recorded in Table 
\ref{tab:Landweber_vs_RBCD} which show that to get the same relative error, larger 
iteration number is required if a larger block number is applied; however, the computational 
times of the RBCD method with $b=2,4,8,16$ are less than that of Landweber iteration (i.e., $b=1$). 

\begin{figure}[htb]
\centering
\subfigure{
\includegraphics[scale=0.45]{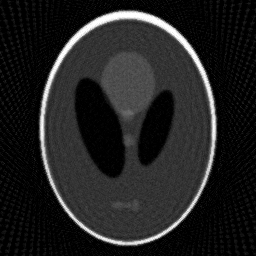} 
}
\subfigure{
\includegraphics[scale=0.45]{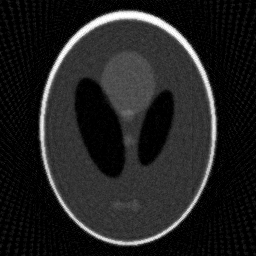} 
}
\subfigure{
\includegraphics[scale=0.45]{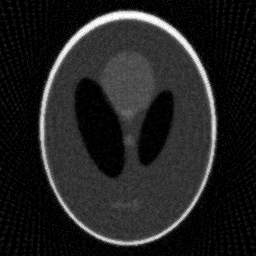}
}
\\
\subfigure{
\includegraphics[scale=0.45]{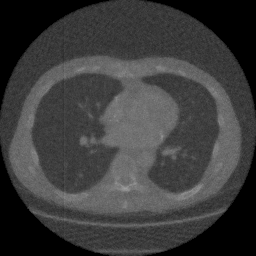}  
}
\subfigure{
\includegraphics[scale=0.45]{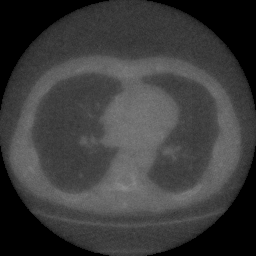} 
}
\subfigure{
\includegraphics[scale=0.45]{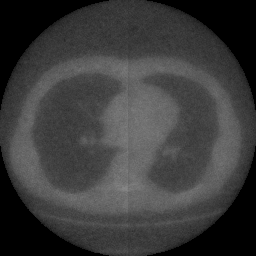} 
}
\caption{The reconstruction results of Algorithm \ref{alg:RBCD_DP} from the sparse view tomographic data with 60 projection angles and three relative noise levels $\delta_{\text{rel}}=0.01$ (left), $0.02$ (middle), and $0.03$ (right) of the Shepp-Logan phantom (top) and real chest CT image (bottom).}
\label{fig:spare_view_data}
\end{figure}

\begin{figure}[htb]
\centering
\subfigure{
\includegraphics[scale=0.45]{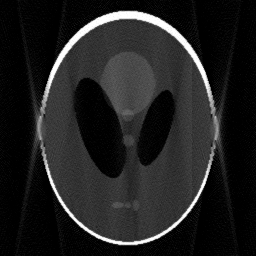} 
}
\subfigure{
\includegraphics[scale=0.45]{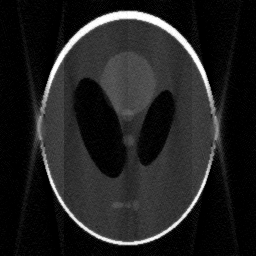} 
}
\subfigure{
\includegraphics[scale=0.45]{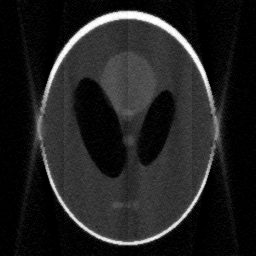}
}
\\
\subfigure{
\includegraphics[scale=0.45]{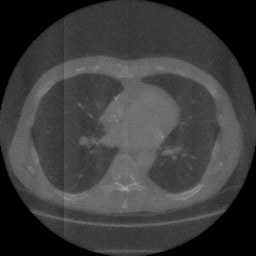}  
}
\subfigure{
\includegraphics[scale=0.45]{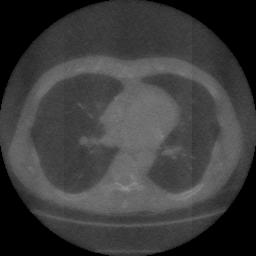} 
}
\subfigure{
\includegraphics[scale=0.45]{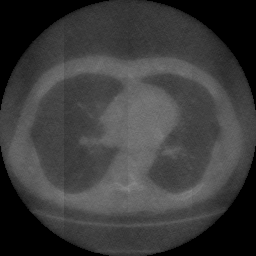} 
}
\caption{The reconstruction results of Algorithm \ref{alg:RBCD_DP} from the limited angle tomographic data with 160 projection angles within $[10^\circ,170^\circ)$ with $1^\circ$ steps and three relative noise levels $\delta_{\text{rel}}=0.01$ (left), $0.02$ (middle), and $0.03$ (right) of the Shepp-Logan phantom (top) and real chest CT image (bottom).}
\label{fig:limited_angle_data}
\end{figure}

It is well-known that if the tomographic data is not complete (here, we only consider the collection for the angular variable of the Radon transform is incomplete), the computed tomography problem is severely ill-posed \cite{L1986}. To test the performance on the ill-posed imaging problem, we use the RBCD method for solving the CT problem by the incomplete view (including limited angle and sparse view cases) tomographic data with distinct three relative noise levels $\delta_{\text{rel}} = 0.01, 0.02$ and $0.03$. The original images here are chosen as the Shepp-Logan phantom and a real chest CT image in Figure \ref{fig:CTimage}. We consider the stimulated spare view tomographic data with 60 projection angles evenly distributed between 1 and 180 degrees and simulated limited angle tomographic data with 160 projection angles within the angular range of $[10^\circ,170^\circ)$ with $1^\circ$ steps, respectively.  In these experiments, we use the Algorithm \ref{alg:RBCD_DP} with the block number $b=4$, $\mu=0.18$ and $\tau=1.1$, and plot the reconstruction results of the sparse view tomography and limited angle tomography in Figure \ref{fig:spare_view_data} and Figure \ref{fig:limited_angle_data}, respectively. It can be shown in Figure \ref{fig:spare_view_data} and \ref{fig:limited_angle_data} that Algorithm \ref{alg:RBCD_DP} can always terminate within a suitable iteration time and thus reconstruct a valid image from the incomplete view tomographic data with relatively small noise (i.e., $\delta_{\text{rel}}$ no more than 0.02 for the Shepp-Logan phantom and 0.01 for the real chest CT image). For the reconstruction results of limited angle tomography, we can observe that the boundary of the image which is tangent to the missing angle in data is hard to recover, and there exists streak artifacts in the reconstruction result. The theoretical reason for these phenomena can be found in \cite{BFJ2018,FQ2013,Q1993}.

\subsection{Compressive temporal imaging}\label{subsect4.2}

In this example, the RBCD method is applied to an application of the video reconstruction, 
the so-called codedaperture compressive temporal imaging \cite{LLY2013}. In this imaging 
modality, one uses a 2D detector to capture the 3D video data in only one snapshot 
measurement by a time-varying mask stack and then reconstructs the video from the coded 
compressed data using the appropriate algorithm.

As described in (\ref{eq:CACTI}), (\ref{eq:CACTI_1}), and (\ref{eq:CACTI_2}) in the 
introduction, the coded aperture compressive temporal imaging in the continuous case 
can be formulated as the general problem (\ref{eq:problem}). To numerically implement 
this application, we need to discretize the continuous model appropriately.

Let $T_i$ be $i$-th temporal interval defined in (\ref{eq:CACTI_1}) such that $|T_i| =1$ 
for $i=1,\ldots,b$. If $b$ is large enough so that the length of each interval $T_i$ is 
short enough relative to the total interval length $|T|$, we may assume that the functions 
$x_i(s,t)$ and $m_i(s,t)$ only depend on the spatial variable $s\in\Omega$, i.e. they 
are time-invariant. Thus, for any $s=(s_1,s_2) \in \Omega $,
$$
m_i(s,t)\equiv m_i^{T_i}(s_1,s_2), \quad 
x_i(s,t)\equiv x_i^{T_i}(s_1,s_2).
$$ 
Denote the video to be reconstructed by $\{x_i^{T_i}\}_{i=1}^{b}$, where $x_i^{T_i}$ is 
the frame within $i$-th time interval and the snapshot measurement by $y$, and suppose 
the mask stack $\{m_i^{T_i}\}_{i=1}^b$ is known. Then the mathematical problem in the 
semi-discrete scheme is to solve the following equation given the measurement $y$ and 
known $m_i^{T_i}$:
\begin{equation}\label{eq:CACTI_3}
    \sum_{i=1}^{b} m_i^{T_i}(s_1,s_2) x_i^{T_i}(s_1,s_2) = y(s_1,s_2).
\end{equation}
Next, we discretize the spatial domain $\Omega$ into a $256\times 256$ pixel grid. The video, 
mask stack, and snapshot measurement are represented by the matrices $\{X_i(m,n)\}_{i=1}^b$,
$\{M_i(m,n)\}_{i=1}^b$, and $Y(m,n)$, $m,n=1,\ldots,256$, respectively. Let 
$\texttt{x}_i\in \mathbb{R}^{256^2\times 1}$ and $\texttt{y}\in \mathbb{R}^{256^2\times 1}$ 
be the vectors obtained by stacking all columns of the matrices $X_i(m,n)$ and $Y(m,n)$.
Define $\texttt{M}_{i}:= \texttt{Diag}(M_i) \in \mathbb{R}^{256^2\times 256^2}$ by arranging 
all the elements of $M_i$ on the diagonal of the matrix $\texttt{M}_{i}$ correspondingly. 
Therefore, equation (\ref{eq:CACTI_3}) is transferred into the fully discrete linear equations
\begin{equation}\label{eq:CACTI_4}
    \sum_{i=1}^b \texttt{M}_i \texttt{x}_i = \texttt{y}.
\end{equation}
In practice, the number of frames $b$ in the video can naturally be chosen as the block number 
for numerical experiments. 

\begin{figure}[htb]
\centering
\subfigure{
\includegraphics[scale=0.32]{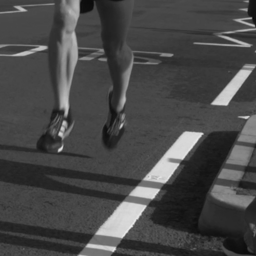} 
}
\subfigure{
\includegraphics[scale=0.32]{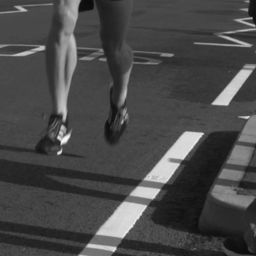}
}
\subfigure{
\includegraphics[scale=0.32]{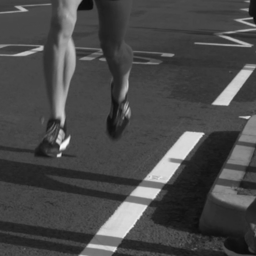}  
}
\subfigure{
\includegraphics[scale=0.32]{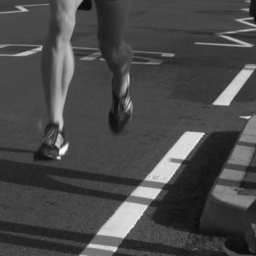} 
}
\\
\subfigure{
\includegraphics[scale=0.32]{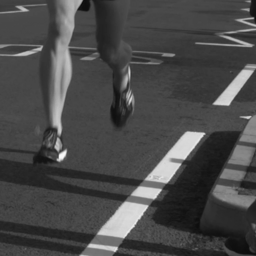} 
}
\subfigure{
\includegraphics[scale=0.32]{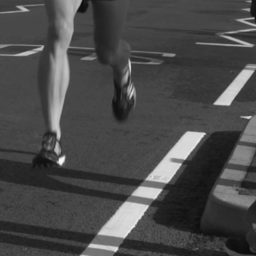} 
}
\subfigure{
\includegraphics[scale=0.32]{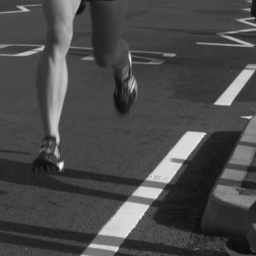}
}
\subfigure{
\includegraphics[scale=0.32]{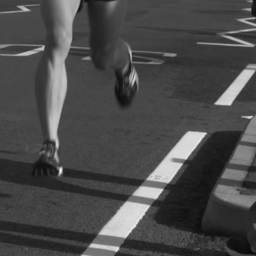}  
}
\caption{The original video of the dataset \texttt{Runner}.}
\label{fig:runner_ture}
\end{figure}

\begin{figure}[htb]
\centering
\subfigure{
\includegraphics[scale=0.32]{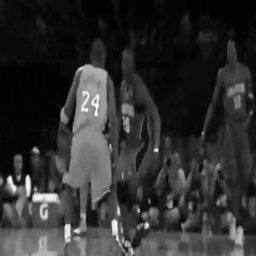} 
}
\subfigure{
\includegraphics[scale=0.32]{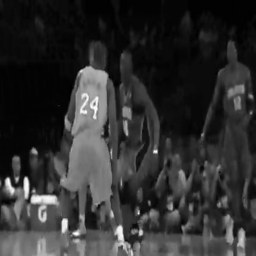}
}
\subfigure{
\includegraphics[scale=0.32]{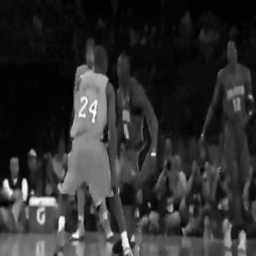}  
}
\subfigure{
\includegraphics[scale=0.32]{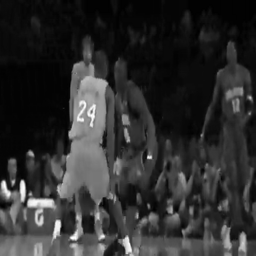} 
}
\\
\subfigure{
\includegraphics[scale=0.32]{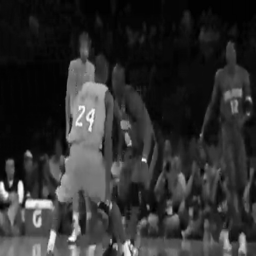} 
}
\subfigure{
\includegraphics[scale=0.32]{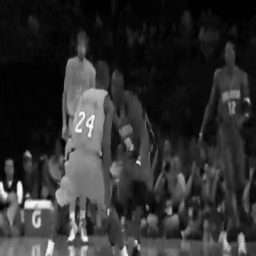} 
}
\subfigure{
\includegraphics[scale=0.32]{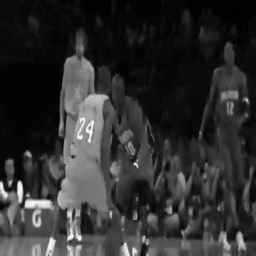}
}
\subfigure{
\includegraphics[scale=0.32]{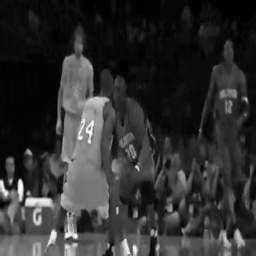}  
}
\caption{The original video of the dataset \texttt{Kobe}.}
\label{fig:Kobe_ture}
\end{figure}

\begin{figure}[!ht]
\centering
\subfigure{
\includegraphics[scale=0.45]{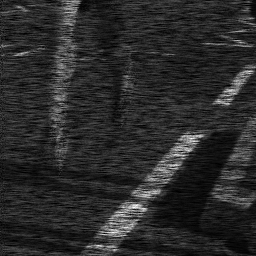} 
}
\subfigure{
\includegraphics[scale=0.45]{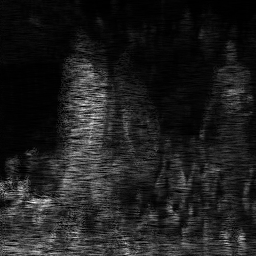}
}
\caption{The exact snapshot measurements of \texttt{Runner} and \texttt{Kobe}.}
\label{fig:video_measurement}
\end{figure}

We assume that the true videos are chosen from the datasets \texttt{Runner} \cite{Rdata} 
and \texttt{Kobe} \cite{YYL2014} with $b=8$ as shown in Figures \ref{fig:runner_ture} 
and \ref{fig:Kobe_ture}. The mask in the stack $\{M_i(m,n)\}_{i=1}^b$ is selected as 
shifting random binary matrices in which the first mask is simulated by a random binary 
atrix with elements drawn from a Bernoulli distribution with parameter $0.5$ and the 
subsequent masks shift one pixel horizontally in order \cite{YYL2014}. Let 
$\texttt{x}^\dagger=(\texttt{x}_1^\dagger,\ldots,\texttt{x}_b^\dagger)$ be the true 
solution and $\texttt{M}=(\texttt{M}_1,\ldots,\texttt{M}_b)$ be the forward matrix with 
$\|\texttt{M}\|^2 = b$. The exact snapshot measurements of the datasets \texttt{Runner} 
and \texttt{Kobe} are given by $\texttt{y}^\dagger :=\sum_{i=1}^b \texttt{M}_i \texttt{x}_i^\dagger$ 
as shown in Figure \ref{fig:video_measurement}. In this application, we need to recover 
$\texttt{x}^\dagger \in \mathbb{R}^{(b\times 256^2)\times 1}$ from the noisy measurement 
$\texttt{y}^\delta \in \mathbb{R}^{256^2\times 1}$ with a relative noise level 
$\d_{\text{rel}} = 0.01$. This implies that the linear system (\ref{eq:CACTI_4}) is 
under-determined due to the compression 
between different frames. Additionally, it is important to note that a significant amount of useful information in each frame of 
the original videos will be lost since nearly half of the elements in every mask are zero.

\begin{figure}[htb]
\centering
\subfigure{
\includegraphics[scale=0.32]{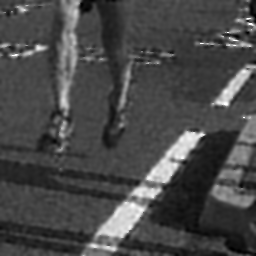} 
}
\subfigure{
\includegraphics[scale=0.32]{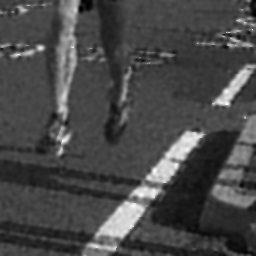}
}
\subfigure{
\includegraphics[scale=0.32]{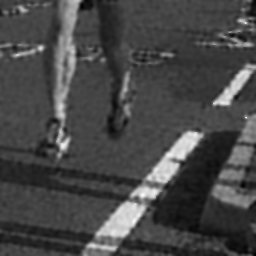}  
}
\subfigure{
\includegraphics[scale=0.32]{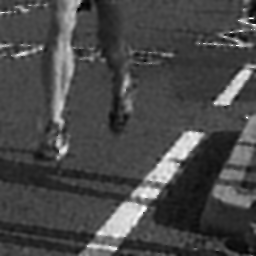} 
}\\
\subfigure{
\includegraphics[scale=0.32]{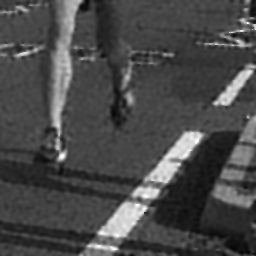} 
}
\subfigure{
\includegraphics[scale=0.32]{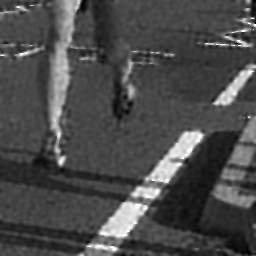} 
}
\subfigure{
\includegraphics[scale=0.32]{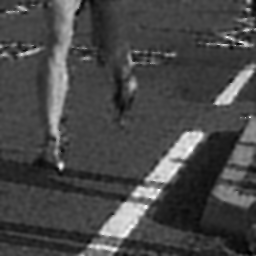} 
}
\subfigure{
\includegraphics[scale=0.32]{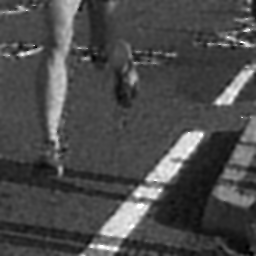} 
}
\caption{The reconstruction results by Algorithm \ref{alg:RBCD+} with $\R$ given by (\ref{TV}) 
from the noisy snapshot measurement with the relative noise level $\delta_{\text{rel}}=0.01$ 
of the dataset \texttt{Runner}.}
\label{fig:runner_reconstruction}
\end{figure}

\begin{figure}[htb]
\centering
\subfigure{
\includegraphics[scale=0.32]{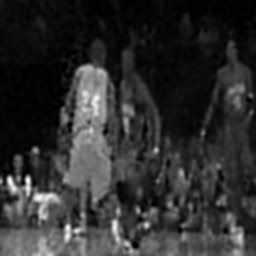} 
}
\subfigure{
\includegraphics[scale=0.32]{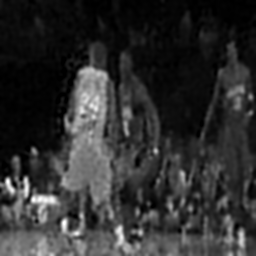}
}
\subfigure{
\includegraphics[scale=0.32]{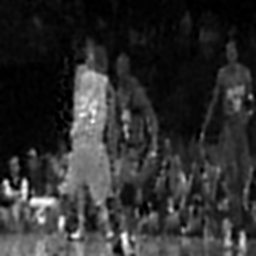}  
}
\subfigure{
\includegraphics[scale=0.32]{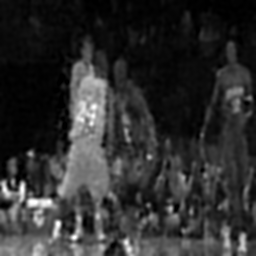} 
}\\
\subfigure{
\includegraphics[scale=0.32]{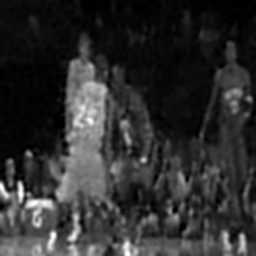} 
}
\subfigure{
\includegraphics[scale=0.32]{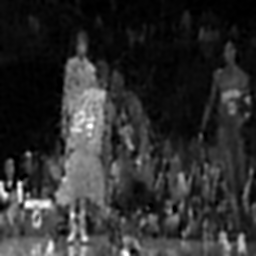} 
}
\subfigure{
\includegraphics[scale=0.32]{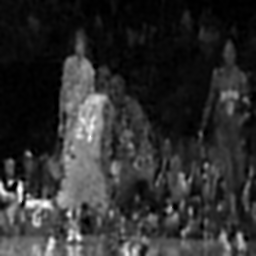} 
}
\subfigure{
\includegraphics[scale=0.32]{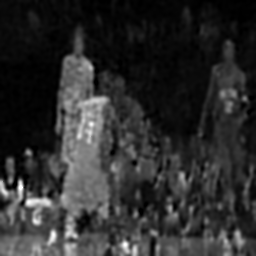} 
}
\caption{The reconstruction results by Algorithm \ref{alg:RBCD+} with $\R$ given by (\ref{TV}) 
from the noisy snapshot measurement with the relative noise level $\delta_{\text{rel}}=0.01$ 
of the dataset \texttt{Kobe}.}
\label{fig:kobe_reconstruction}
\end{figure}

To solve this ill-posed video reconstruction problem, we incorporate the piece-wise
constancy feature of the video into our proposed method, Algorithm \ref{alg:RBCD+}. To this 
end, we take 
\begin{align}\label{TV}
\R(\texttt{x}) = \sum_{i=1}^b \R_i(\texttt{x}_i) \quad \mbox{with } 
\R_i(\texttt{x}_i):= \frac{1}{2} \| \texttt{x}_i \|^2 + \lambda |\texttt{x}_i|_{TV},
\end{align}
where $\lambda>0$ is a positive constant and $|\texttt{x}_i|_{TV}$ denotes the 2-dimensional 
isotropic total variation of $\texttt{x}_i$ for each frame; see \cite{BT2009}. It is 
clear that $\R$ satisfies Assumption \ref{ass2} with $\kappa = 1/2$. In order to implement 
Algorithm \ref{alg:RBCD+}, we need to update $\texttt{x}_{k+1, i_k}^\d$ from $\xi_{k+1,i_k}^\d$
which requires to solving a total variation denoising problem of the form  
$$
\texttt{x}_{k+1,i_k}^\d = \arg\min_z \left\{\R_{i_k}(z) - \l \xi_{k+1, i_k}^\d, z\r\right\}
=\arg\min_z\left\{\lambda |z|_{TV} + \frac{1}{2} \|z - \xi_{k+1, i_k}^\d\|^2\right\}.
$$
In the experiment, this sub-problem is solved by the iterative clipping algorithm \cite{BT2009}.
When executing Algorithm \ref{alg:RBCD+}, we use $\xi_0 =0$ and $\gamma = 2\kappa \mu/b$ with 
$\mu=1.99$ and for the datasets \texttt{Runner} and \texttt{Kobe} we use $\lambda = 15$ and $30$
respectively. We run the algorithm for $1500$ iteration steps and report the reconstruction 
results in Figure \ref{fig:runner_reconstruction} and \ref{fig:kobe_reconstruction}. The 
detailed numerical results of the above experiments are recorded in Table \ref{tab:CACTI_reconstruction}, 
including the computational time, the peak signal-to-noise ratio (PSNR), the structural similarity 
index measure (SSIM) \cite{WBS2004}, and the relative error. These results demonstarte that, with 
a TV denoiser in the domain of prime variable, Algorithm \ref{alg:RBCD+} can recover a great deal 
of missing information from the compressed snapshot measurement and thus reconstruct a good 
approximation of the original video within an acceptable amount of time even if the data is 
corrupted by noise.

\begin{table}[htb]
    \centering
    \begin{tabular}{lllllll}
    \toprule
         &  Time (s)  & PSNR (db) & SSIM & relative error \\
    \midrule   
        \texttt{Runner} & 23.5595 & 27.8292 & 0.8012  & 0.0154  \\
        \texttt{Kobe} & 23.6561 & 28.3153 & 0.8883 & 0.0216 \\      
    \bottomrule
    \end{tabular}
    \caption{Numerical results of Algorithm \ref{alg:RBCD+} with $\R$ given by (\ref{TV}) 
    for the datasets \texttt{Runner} and \texttt{Kobe} using the noisy snapshot measurement 
    with the relative noise level $\delta_{\text{rel}}=0.01$.}
    \label{tab:CACTI_reconstruction}   
\end{table}

%If we consider the Algorithm \ref{alg:RBCD+TV} combined with the discrepancy principle, it will lead to the following method.
%\begin{algorithm}[RBCD-TV method with the discrepancy principle for solving (\ref{eq:CACTI_4})]\label{alg:RBCD+TV_DP}

%Let $\xi_0 = 0 \in \X$ and {\rm$\texttt{x}_0 = 0$} as two initial guesses.
%Set $\xi_0^\d := \xi_0$, {\rm$\texttt{x}_0^\d := \texttt{x}_0$} and calculate {\rm$r_0^\d:= \texttt{M} \texttt{x}_{0}^\d - \texttt{y}^\d$}. 
%Choose $\tau>1$ and $\gamma= 2c_0\mu/b$, where $\mu \in (0,2)$. For $k \ge 0$ 
%do the following:  
%\begin{enumerate}
%\item[(i)] If $\|r_k^\d\|\leq \tau \delta$, stop and output {\rm$\texttt{x}_{k}^\d$}.
%
%\item[(ii)] Pick an index $i_k \in \{1, \cdots, b\}$ randomly via the uniform distribution; 
%
%\item[(iii)] Update $\xi_{k+1}^\d$ by the equation (\ref{RBCD.72}), i.e. setting $\xi_{k+1,i}^\d = \xi_{k,i}^\d$ 
%for $i \ne i_k$ and 
%{\rm
%$$
%\xi_{k+1, i_k}^\d = \xi_{k, i_k}^\d - \gamma \texttt{M}_{i_k} r_k^\d;
%$$
%}
%
%\item[(iv)] Update {\rm$\texttt{x}_{k+1}^\d$} by the equation (\ref{RBCD.73}), i.e setting {\rm$\texttt{x}_{k+1,i}^\d = \texttt{x}_{k,i}^\d$} for $i \ne i_k$ and 
%{\rm
%$$
%\texttt{x}_{k+1,i_k}^\d =\arg\min_{z\in X_{i_k}} \left\{ \frac{1}{2} \left\| z - \xi_{k+1,i_k}^\d  \right\|^2 + \lambda |z|_{TV}\right\};    
%$$
%}
%
%\item[(v)] Calculate {\rm$r_{k+1}^\d = r_k^\d + \texttt{M}_{i_k}(\texttt{x}_{k+1, i_k}^\d - \texttt{x}_{k,i_k}^\d)$}. 
%\end{enumerate}
%\end{algorithm}

Considering the same problem for \texttt{Runner} and \texttt{Kobe} using the noisy data
with $\delta_{\text{rel}}=0.01$, we run the Algorithm \ref{alg:RBCD+DP} with $\tau=2$, $\mu=1.99$, 
and $\R$ given by (\ref{TV}).We use $\lambda = 15$ for the datasets \texttt{Runner} and 
$\lambda = 30$ for the dataset \texttt{Kobe}. This algorithm terminates the iteration based 
on the {\it a posteriori} discrepancy principle. The reconstruction results are shown in 
Figures \ref{fig:runner_reconstruction_DP} and \ref{fig:kobe_reconstruction_DP}. We also 
record the stopping index $k_\delta$, PSNR, SSIM, and relative error of the reconstructions 
in Table \ref{tab:CACTI_reconstruction_DP}. 
It can be observed that Algorithm \ref{alg:RBCD+DP} consistently terminates the iterations 
within a finite number of steps and provides efficient reconstructions of the true videos, 
even when the data is corrupted by noise. Therefore, Algorithm \ref{alg:RBCD+DP} effectively 
meets the key requirement of automatic stopping for this large-scale video reconstruction 
in practice.
\begin{figure}[htb]
\centering
\subfigure{
\includegraphics[scale=0.32]{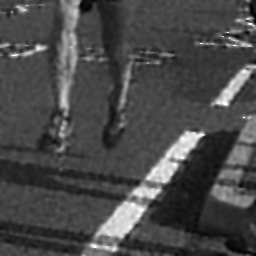} 
}
\subfigure{
\includegraphics[scale=0.32]{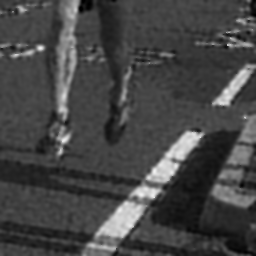}
}
\subfigure{
\includegraphics[scale=0.32]{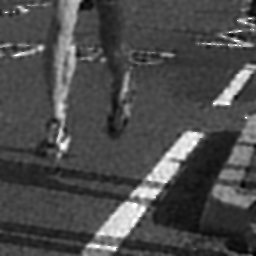}  
}
\subfigure{
\includegraphics[scale=0.32]{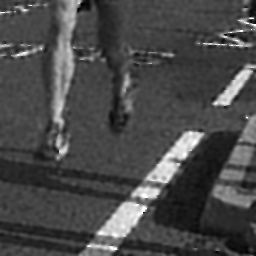} 
}\\
\subfigure{
\includegraphics[scale=0.32]{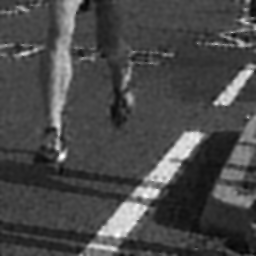} 
}
\subfigure{
\includegraphics[scale=0.32]{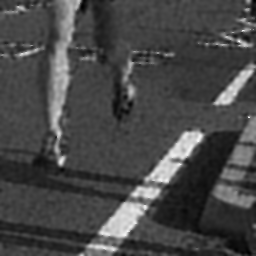} 
}
\subfigure{
\includegraphics[scale=0.32]{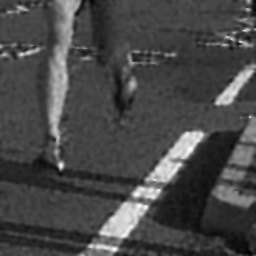} 
}
\subfigure{
\includegraphics[scale=0.32]{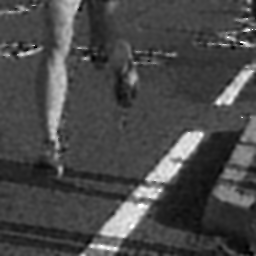} 
}
\caption{The reconstruction results by Algorithm \ref{alg:RBCD+DP} with $\R$ given by 
(\ref{TV}) from the noisy snapshot measurement with the relative noise level 
$\delta_{\text{rel}}=0.01$ of the dataset \texttt{Runner}.}
\label{fig:runner_reconstruction_DP}
\end{figure}

\begin{figure}[htb]
\centering
\subfigure{
\includegraphics[scale=0.32]{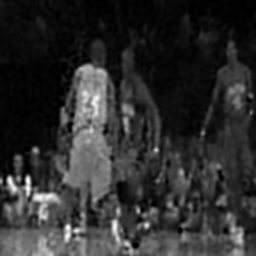} 
}
\subfigure{
\includegraphics[scale=0.32]{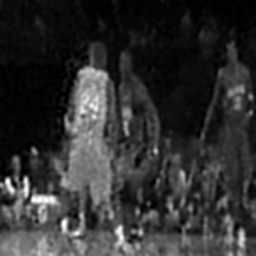}
}
\subfigure{
\includegraphics[scale=0.32]{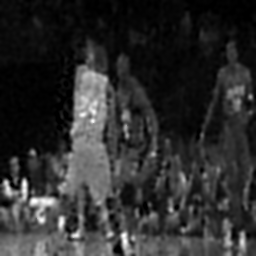}  
}
\subfigure{
\includegraphics[scale=0.32]{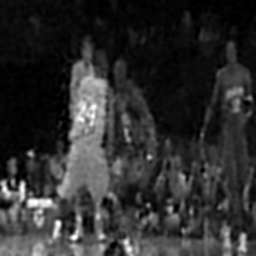} 
}\\
\subfigure{
\includegraphics[scale=0.32]{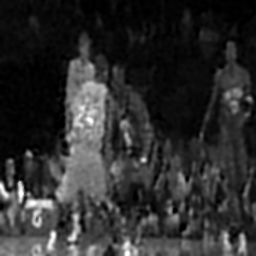} 
}
\subfigure{
\includegraphics[scale=0.32]{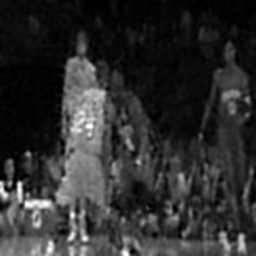} 
}
\subfigure{
\includegraphics[scale=0.32]{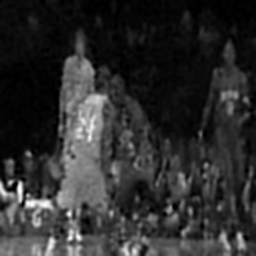} 
}
\subfigure{
\includegraphics[scale=0.32]{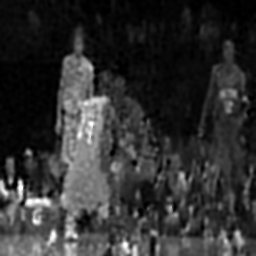} 
}
\caption{The reconstruction results by Algorithm \ref{alg:RBCD+DP} with $\R$ given by 
(\ref{TV}) from the noisy snapshot measurement with the relative noise level 
$\delta_{\text{rel}}=0.01$ of the dataset \texttt{Kobe}.}
\label{fig:kobe_reconstruction_DP}
\end{figure}

\begin{table}[htb]
    \centering
    \begin{tabular}{lllllll}
    \toprule
         & $k_\delta$  & PSNR (db) & SSIM & relative error \\
    \midrule   
        \texttt{Runner} & 1306 & 27.5785 & 0.7983  & 0.0163  \\
        \texttt{Kobe} & 2611 & 28.4458 & 0.8842 & 0.0209 \\      
    \bottomrule
    \end{tabular}
    \caption{Numerical results of Algorithm \ref{alg:RBCD+DP} with $\R$ gicen by (\ref{TV})
    for the datasets \texttt{Runner} and \texttt{Kobe} using the noisy snapshot measurement 
    with the relative noise level $\delta_{\text{rel}}=0.01$.}
    \label{tab:CACTI_reconstruction_DP}   
\end{table} 

\section{\bf Conclusions}

In this paper, we addressed linear ill-posed inverse problems of the separable form 
$\sum_{i=1}^{b} A_i x_i =y$ in Hilbert spaces and proposed a randomized block 
coordinate descent (RBCD) method for solving such problems. Although the RBCD method 
has been extensively studied for well-posed convex optimization problems, there has 
been a lack of convergence analysis for ill-posed problems. In this work, we 
investigated the convergence properties of the RBCD method with noisy data under both 
\textit{a priori} and \textit{a posteriori} stopping rules. We proved that the RBCD 
method, when combined with an \textit{a priori} stopping rule, serves as a convergent 
regularization method in the sense of weak convergence almost surely. Additionally, 
we explored the early stopping of the RBCD method, demonstrating that the discrepancy 
principle can terminate the iteration after a finite number of steps almost surely. 
Furthermore, we developed a strategy to incorporate convex regularization terms into 
the RBCD method to enhance the detection of solution features. To validate the 
performance of our proposed method, we conducted numerical simulations, demonstrating 
its effectiveness.

It should be noted that when the sought solution is smooth and decomposed into many 
small blocks artificially, the numerical results obtained by the RBCD method, which 
are not reported here, may not meet expectations due to the mismatch between adjacent 
blocks. Therefore, there is significant interest in developing novel strategies to 
overcome this mismatch between adjacent blocks, thereby improving the performance of 
the RBCD method. However, in cases where the object to be reconstructed naturally 
consists of many separate frames, such as the coded aperture compressive temporal imaging, 
the RBCD method demonstrates its effectiveness and yields satisfactory results.

\section*{\bf Acknowledgement}

\noindent
The work of Q. Jin is partially supported by the Future Fellowship of the Australian Research Council (FT170100231).
The work of D. Liu is supported by the China Scholarship Council program (Project ID: 202307090070).


\begin{thebibliography}{50}	
\bibitem{BT2009} A. Beck and M. Teboulle, {\it Fast gradient-based algorithms for constrained total variation image denoising and deblurring problems}, IEEE Trans. Image Process., 18 (2009), pp.2419--2434.

\bibitem{B2011} D. P. Bertsekas, {\it Incremental proximal methods for large scale convex optimization}, Math. Program., 129 (2011), pp. 163--195.

\bibitem{BFJ2018} L. Borg, J. Frikel, J. S. J\o{}rgensen and E. T. Quinto, {\it Analyzing reconstruction artifacts from arbitrary incomplete X-ray CT data}, SIAM J. Imaging Sci., 11 (2018), pp. 2786--2814.

\bibitem{B2020} P. Br\'{e}maud, {\it Probability theory and stochastic processes}, Universitext, 
Springer, Cham, 2020. 

\bibitem{CP2015} P. L. Combettes and J.-C. Pesquet, {\it Stochastic quasi-Fej\'{e}r block-coordinate fixed point iterations with random sweeping}, SIAM J. Optim., 25 (2015), pp. 1221--1248.

%\bibitem{CHLS2008} A. De Cezaro, M. Haltmeier, A. Leitão, and O. Scherzer, {\it On  steepest-descent-Kaczmarz methods  for  regularizing  systems  of  nonlinear  ill-posed equations}, Appl. Math. Comput., 202 (2008), pp. 596--607.

\bibitem{EHN1996} H. W. Engl, M. Hanke and A. Neubauer, {\it Regularization of Inverse Problems}, 
Dordrecht, Kluwer, 1996.

\bibitem{FQ2013} J. Frikel and E. T. Quinto, {\it Characterization and reduction of artifacts in limited angle tomography}, Inverse Problems, 29 (2013), 125007.

%\bibitem{G1984} C. W. Groetsch, {\it The theory of Tikhonov regularization for Fredholm equations of the first kind}, Research Notes in Mathematics, 105. Pitman (Advanced Publishing Program), Boston, MA, 1984.

\bibitem{HLS2007} M. Haltmeier, A. Leitão and O. Scherzer, {\it Kaczmarz methods for regularizing nonlinear ill-posedequations. I. Convergence analysis}, Inverse Probl. Imaging, 1 (2007), pp. 289--298.

\bibitem{HS2012} P. C. Hansen and M. Saxild-Hansen, {\it AIR tools—a MATLAB package of algebraic iterative reconstruction methods}, J. Comput. Appl. Math., 236 (2012), pp. 2167--2178.

%\bibitem{JL2019} B. Jin and X. Lu, {\it On the regularizing property of stochastic gradient descent}, Inverse Problems, 35 (2019), 015004.

\bibitem{Jin2016} Q. Jin, {\it Landweber-Kaczmarz method in Banach spaces with inexact inner solvers}, Inverse Problems, 32 (2016), 104005.

\bibitem{JLZ2023} Q. Jin, X. Lu and L. Zhang, {\it Stochastic mirror descent method for linear ill-posed in Banach spaces}, Inverse Problems, 39 (2023), 065010.

\bibitem{JW2013}
Q. Jin and W. Wang, {\it Landweber iteration of Kaczmarz type with general non-smooth convex penalty functionals},
Inverse Problems, 29 (2013), no. 8, 085011, 22 pp.

\bibitem{KS2002} R. Kowar and O. Scherzer, {\it Convergence analysis of a Landweber-Kaczmarz method for solving
nonlinear ill-posed problems}, Ill-posed and inverse problems, 253–270, VSP, Zeist, 2002.

\bibitem{LQJ2017} C. Liu, J. Qiu, and M. Jiang, {\it Light field reconstruction from projection modeling of focal stack}, Opt. Express, 25 (2017), pp. 11377--11388. 

\bibitem{LLY2013} P. Llull, X. Liao, X. Yuan, J. Yang, D. Kittle, L. Carin, G. Sapiro, and D. J. Brady, {\it Coded aperture compressive temporal imaging}, Opt. Express, 21 (2013), pp. 10526--10545.

\bibitem{L1986} A. K. Louis, {\it Incomplete data problems in X-ray computerized tomography. I. singular value decomposition of the limited angle transform}, Numer. Math., 48 (1986), pp. 251--262.

\bibitem{LX2015} Z. Lu and L. Xiao, {\it On the complexity analysis of randomized block-coordinate descent methods}, Math. Program., 152 (2015), 615--642.

\bibitem{N2001} F. Natterer, {\it The Mathematics of Computerized Tomography}, SIAM, Philadelphia, 2001.

\bibitem{N2012} Y. Nesterov, {\it Efficiency of coordinate descent methods on huge-scale optimization problems}, SIAM J. Optim., 22 (2012), no. 2, pp. 341--362.

%\bibitem{P1975} A. Papoulis, {\it A new algorithm in spectral analysis and band-limited extrapolation}, IEEE Trans. Circuit Syst., 22 (1975), pp. 735--742.

\bibitem{Q1993} E. T. Quinto, {\it Singularities of the X-ray transform and limited data tomography in $\mathbb{R}^2$ and $\mathbb{R}^3$}, SIAM J. Appl. Math., 24 (1993), pp. 1215--1225.

\bibitem{RNH2019} S. Rabanser, L. Neumann and M. Haltmeier, 
{\it Analysis of the block coordinate descent method for linear ill-posed problems}, 
SIAM J. Imaging Sci., 12 (2019), no. 4, pp. 1808--1832.

\bibitem{RT2014} P. Richt\'{a}rik and M. Tak\'{a}{c}, {\it Iteration complexity of randomized block-coordinate descent methods for minimizing a composite function}, Math. Program.,  144 (2014), 1--38.

\bibitem{Rdata} {\it Runner data}, https://www.videvo.net/video/elite-runner-slow-motion/4541/.

\bibitem{ST2013} A. Saha and A. Tewari, {\it On the nonasymptotic convergence of cyclic coordinate descent 
methods}, SIAM J. Optim., 23 (2013), pp. 576--601.

\bibitem{T2001} P. Tseng, {\it Convergence of a block coordinate descent method for nondifferentiable minimization}, J. Optim. Theory Appl., 109 (2001), 475--494.

\bibitem{TY2009a} P. Tseng and S. Yun, {\it Block-coordinate gradient descent method for linearly constrained nonsmooth separable optimization}, J. Optim. Theory Appl. 140 (2009), 513--535.

\bibitem{TY2009b} P. Tseng and S.  Yun, {\it A coordinate gradient descent method for nonsmooth separable minimization}, Math. Program., 117 (2009), 387--423.

%\bibitem{VA2008} V. Vaish and A. Adams, {\it The (new) Stanford light field archive}, Comput. Graph. Lab. Stanf. Uni., 6 (2008).

\bibitem{WJW2008} A. Wagadarikar, R. John, R. Willett and D. J. Brady, {\it Single disperser design for coded aperture snapshot spectral imaging}, Appl. Opt., 47 (2008), pp. B44--B51.

\bibitem{WBS2004} Z. Wang, A. C. Bovik, H. R. Sheikh and E. P. Simoncelli, {\it Image quality assessment: from error visibility to structural similarity}, IEEE Trans. Image Process., 13 (2004), pp. 600--612.

\bibitem{W2015} S. J. Wright, {\it Coordinate descent algorithms}, Math. Program., 151 (2015), pp. 3--34.

\bibitem{YYL2014} J. Yang, X. Yuan, X. Liao, P. Llull, G. Sapiro, D. J. Brady and L. Carin, {\it Video compressive sensing using Gaussian mixture models}, IEEE Trans. Image Process., 23 (2014), pp. 4863--4878.

\bibitem{YWL2016} X. Yin, G. Wang, W. Li and Q. Liao, {\it Iteratively reconstructing 4D light fields from focal stacks}, Appl. Opt., 55 (2016), pp. 8457--8463.

\bibitem{YBK2021} X. Yuan, D. J. Brady and A. K. Katsaggelos, {\it Snapshot compressive imaging: theory, algorithms, and applications}, IEEE Signal Process. Mag., 38 (2021), pp.65--88.

\bibitem{Z2002} C. Z\u{a}linscu, {\it Convex Analysis in General Vector Spaces}, World Scientific Publishing Co.,
Inc., River Edge, New Jersey, 2002.

\end{thebibliography}
\end{document}